\documentclass[11pt,letterpaper]{article}

\usepackage[utf8]{inputenc}
\usepackage{comment}
\usepackage{microtype}
\usepackage{graphicx}
\usepackage{booktabs} 
\usepackage{scalerel}
        \usepackage{setspace}

\usepackage[margin=1.04in]{geometry}
\usepackage[style=alphabetic, maxbibnames=15, maxcitenames=15, natbib=true,
  maxalphanames=10, backend=biber, sorting=nty, backref=true]{biblatex}
\DefineBibliographyStrings{english}{backrefpage = {page},backrefpages = {pages},}
\usepackage[T1]{fontenc}    
\usepackage{hyperref}       
\usepackage{url}            
\usepackage{amsfonts}       
\usepackage{nicefrac}       
\usepackage{microtype}      
\usepackage{xcolor}         
\usepackage{amsmath}
\usepackage{amsthm}
\usepackage{algorithm} 
\usepackage{algpseudocode}
\usepackage{subcaption}
\usepackage{arydshln}
\usepackage[font=small,labelfont=bf]{caption}

\usepackage{needspace}

\input{mySymbol.sty}

\newtheorem{assumption}{\hspace{0pt}\bf Assumption}
\newtheorem{lemma}{\hspace{0pt}\bf Lemma}
\newtheorem{proposition}{\hspace{0pt}\bf Proposition}

\newtheorem{theorem}{\hspace{0pt}\bf Theorem}

\newtheorem{remark}{\hspace{0pt}\bf Remark}

\newtheorem{definition}{\hspace{0pt}\bf Definition}

\usepackage{hyperref}
\hypersetup{colorlinks,citecolor=blue,linktocpage,breaklinks=true}
\begin{document}

\title{Limited-Memory Greedy Quasi-Newton Method with Non-asymptotic Superlinear Convergence Rate}

\author{Zhan Gao\thanks{Department of Computer Science and Technology, University of Cambridge, Cambridge, UK\qquad\{zg292@cam.ac.uk\}}         \and
        Aryan Mokhtari\thanks{Department of Electrical and Computer Engineering, The University of Texas at Austin, Austin, TX, USA\qquad\{mokhtari@austin.utexas.edu\}}
\and
        Alec Koppel\thanks{JP Morgan Chase AI Research, New York, NY, USA ~ \{alec.koppel@jpmchase.com\}}
}

\date{}

\maketitle

\setstretch{1.02}

\begin{abstract}
Non-asymptotic convergence analysis of quasi-Newton methods has gained attention with a landmark result establishing an explicit local superlinear rate of $\mathcal{O}((1/\sqrt{t})^t)$. The methods that obtain this rate, however, exhibit a well-known drawback: they require the storage of the previous Hessian approximation matrix or 
all past curvature information to form the current Hessian inverse approximation. Limited-memory variants of quasi-Newton methods such as the celebrated L-BFGS alleviate this issue by leveraging a limited window of past curvature information to construct the Hessian inverse approximation. As a result, their per iteration complexity and storage requirement is $\mathcal{O}(\tau d)$ where $\tau \le d$ is the size of the window and $d$ is the problem dimension reducing the $\mathcal{O}(d^2)$ computational cost and memory requirement of standard quasi-Newton methods. However, to the best of our knowledge, there is no result showing a non-asymptotic superlinear convergence rate for any limited-memory quasi-Newton method. In this work, we close this gap by presenting a Limited-memory Greedy BFGS (LG-BFGS) method that can achieve an explicit non-asymptotic superlinear rate. We incorporate displacement aggregation, i.e., decorrelating projection, in post-processing gradient variations, together with a basis vector selection scheme on variable variations, which \emph{greedily} maximizes a progress measure of the Hessian estimate to the true Hessian. Their combination allows past curvature information to remain in a sparse subspace while yielding a valid representation of the full history. Interestingly, our established \textit{non-asymptotic} superlinear convergence rate demonstrates an explicit trade-off between the convergence speed and memory requirement, which to our knowledge, is the first of its kind. Numerical results corroborate our theoretical findings and demonstrate the effectiveness of our method.
\end{abstract}

\newpage

\section{Introduction}

This work focuses on the minimization of a smooth and strongly convex function as follows: 
\begin{equation}\label{eq:ERMproblem1}
		\min_{\bbx} f(\bbx) 
\end{equation} 
where $\bbx\in\mathbb{R}^d$ and $f:\mathbb{R}^d\rightarrow \mathbb{R}$ is twice differentiable. It is well-known that gradient descent and its accelerated variant can 
converge to the optimal solution at a linear rate \citep{nesterov2003introductory}. While this is advantageous from the perspective of simplicity and low complexity of $\mathcal{O}(d)$, the convergence rate for these methods exhibits dependence on the conditioning of the problem. Therefore, when the problem is ill-conditioned, such methods could be slow. Second-order methods such as Newton's method or cubic regularized Newton method are able to address the ill-conditioning issue by leveraging the objective function curvature. Alas, their complexity of $\mathcal{O}(d^3)$ limits their application.  

Quasi-Newton techniques strike a balance between gradient methods and second-order algorithms in terms of having a complexity of $\mathcal{O}(d^2)$, while refining the dependence on the problem conditioning and achieving local superlinear rates  \citep{byrd1987global,gao2019quasi}. They approximate the Hessian or the Newton direction in a variety of ways leading to different algorithms, such as Broyden's method \citep{broyden1965class,broyden1973local,gay1979some}, the celebrated Broyden \citep{broyden1967quasi,broyden1973local}-Fletcher\citep{fletcher1963rapidly}-Goldfarb\citep{goldfarb1970family}-Shanno\citep{shanno1970conditioning} (BFGS) method, as well as Davidon-Fletcher-Powell (DFP) algorithm \citep{fletcher1963rapidly,davidon1991variable} and Symmetric Rank 1 (SR1) method \citep{conn1991convergence}. For some time, it has been known that these methods can achieve asymptotic local superlinear rates \citep{dennis1974characterization}. More recently, the explicit non-asymptotic local rate has been characterized for Broyden class when combined with a greedy basis vector selection \citep{rodomanov2021greedy,lin2021greedy}, and subsequently, the explicit rates have been derived for standard BFGS \citep{rodomanov2021rates,jin2022non} and SR1 \citep{sr1_superlinear}. While these results are promising in clarifying the performance trade-offs between different Hessian approximation methodologies, there are important limitations. In particular, their per-update complexity is quadratic $\mathcal{O}(d^2)$ 
which could be costly in high-dimensional settings. Also, they require to store either the most recent Hessian inverse approximation on 
a memory of $\mathcal{O}(d^2)$ or all past differences of gradients and iterates, known as curvature information, on 
a memory of $\mathcal{O}(td)$. Hence, in the case that the iteration index $t$ becomes large, their storage requirement becomes quadratic in dimension. 

Limited-memory BFGS (L-BFGS) drops all but the past $\tau$ curvature pairs, alleviating these bottlenecks and reducing the iteration complexity and memory requirement to $\mathcal{O}(\tau d)$. Note that $\tau$ is often a constant much smaller than the 
dimension $d$. However, previous analyses of classic L-BFGS 
only guarantee either a sublinear \citep{mokhtari2015global,boggs2019adaptive,yousefian2020stochastic} or a linear \citep{moritz2016linearly} convergence rate, which are no better than gradient descent. Recently, \cite{berahas2022limited} showed that when L-BFGS is combined with a way to project removed information onto the span that remains, it asymptotically converges at a superlinear rate. However, they failed to characterize the region for which the superlinear rate is achieved, the explicit rate of convergence, or the role of memory size. 
Thus, we pose the following questions:
\smallskip
\smallskip
\smallskip
\renewenvironment{quote}{%
  \list{}{%
    \leftmargin1cm   
    \rightmargin\leftmargin
  }
  \item\relax
}
{\endlist}
\begin{quote}
\vspace{-2mm}
   \textit{Can we identify a class of problems for which a fixed-size limited-memory quasi-Newton method may achieve a non-asymptotic superlinear rate? If yes, how would the memory size impact the convergence rate?} 
\end{quote}
\smallskip
\textbf{Contributions.} In this work, we provide an affirmative answer to the first question by synthesizing a Limited-memory Greedy (LG)-BFGS method that employs the greedy basis vector selection 
for the variable variation \citep{rodomanov2021greedy} to construct the Newton direction, but along a subspace whose rank is defined by the memory size. Different from the prior instantiation of displacement aggregation \citep{berahas2022limited,sahu2023modified}, we specify the basis for subspace projection to be the one identified by the aforementioned greedy basis selection. 
This innovative design allows 
achieving the approximation attributes of greedy BFGS but along a low-dimensional subspace only, under appropriate structural assumptions on the curvature of limited-memory Hessian approximation error. These ingredients together enable us to generalize the convergence theory of greedy BFGS to the limited-memory setting and to 
provide the first non-asymptotic local superlinear rate for a quasi-Newton method with 
affordable storage requirements. We further answer the second question by characterizing the role of memory size in the superlinear convergence 
rate of 
LG-BFGS. More precisely, we establish an explicit trade-off between the memory size and the contraction factor that appears in the superlinear rate, and corroborate this trade-off in numerical experiments. 
All proofs are summarized in the appendix.

%
%
%
\smallskip
\noindent {\bf Related Work.}  Alternative approach for  memory reduction is randomized sub-sampling (sketching) of Hessian matrices  \citep{pilanci2017newton,gower2017randomized,gower2019rsn}. This distinct line of inquiry can establish local quadratic rates of convergence when the memory associated with the sub-sampling size achieves a threshold  called effective problem dimension \citep{lacotte2021adaptive}; however, ensuring this condition in practice is elusive, which leads to nontrivial parameter tuning issues in practice \citep{NEURIPS2021_16837163}. Efficient approximation criteria for sketching is an ongoing avenue of research, see e.g., \citep{derezinski2020precise}.

\smallskip
\noindent {\bf Notation.} The weighted norm of $\bbx$ by  a positive defnite matrix $\bbZ$ is denoted by  $\|\bbx\|_\bbZ = \sqrt{\bbx^\top \bbZ\bbx}$.

\section{Problem Setup and Background}

This paper focuses on 
quasi-Newton (QN) algorithms for 
problem \eqref{eq:ERMproblem1}, which aim at approximating the objective function curvature. 
Let $\bbx_t$ be the decision variable, $\nabla f(\bbx_t)$ the gradient, and $\bbB_t$ the Hessian approximation at iteration $t$. The update rule for most QN methods follows the recursion 
\begin{align}\label{eq:Newton}
	\bbx_{t+1}  = \bbx_t - \alpha \bbH_t \nabla f(\bbx_t)
\end{align} 
where $\alpha$ is the step-size and  $\bbH_t$ is the inverse of the Hessian approximation matrix $\bbB_t$. By updating 
$\bbB_t$ with different schema, one may obtain a variety of QN methods \citep{davidon1991variable}. In the case of Broyden-Fletcher-Goldfarb-Shanno (BFGS), we define the variable variation $\bbs_t = \bbx_{t+1} - \bbx_t$ and gradient variation $\bbr_t = \nabla f(\bbx_{t+1}) - \nabla f(\bbx_t)$, and sequentially select the matrix close to the previous estimate, while satisfying the secant condition $\bbB\bbs_t =\bbr_t$. This would lead to the following update 
\begin{align}\label{eq:BFGS}
	\bbB_{t+1} = \bbB_t + \frac{\bbr_t \bbr_t^\top}{\bbr_t^\top \bbs_t} - \frac{\bbB_t \bbs_t \bbs_t^\top \bbB_t^\top}{\bbs_t^\top \bbB_t \bbs_t}.
\end{align}
Given the rank-two structure of the above update, one can use the Sherman-Morrison formula to establish a simple update for the Hessian inverse approximation matrix which is given by  
\begin{align}\label{eq:InverseBFGS}
	\bbH_{t+1} =  \Big(\bbI - \frac{\bbr_t \bbs_t^\top}{\bbs_t^\top \bbr_t}\Big)^\top\  \bbH_t\ \Big(\bbI - \frac{\bbr_t\bbs_t^\top}{\bbs_t^\top \bbr_t}\Big) + \frac{\bbs_t \bbs_t^\top}{\bbs_t^\top \bbr_t}  .
\end{align}
The computational cost of this update 
is $\ccalO(d^2)$ as it requires matrix-vector product calculations. Moreover, it requires a memory of $\ccalO(d^2)$ to store the matrix $\bbH_t$. Using the recursive structure of the update in \eqref{eq:InverseBFGS} one can also implement BFGS at the computational cost and memory of $\ccalO(td)$ which is beneficial 
in the regime that $t\leq d$, while 
this gain disappears as soon as $t>d$. 
More precisely, suppose we are at iteration $t$ and we have access to all previous iteration and gradient differences $\{(\bbs_u, \bbr_u)\}_{u=0}^{t-1}$, referred to as the \emph{curvature pairs}, as well as the initial Hessian inverse approximation matrix $\bbH_0$ that is diagonal. The memory requirement for storing the above information is $\ccalO(td)$. It can be shown that the BFGS descent direction $-\bbH_t \nabla f(\bbx_t)$ can be computed at an overall cost of $\ccalO(td)$ with the well-known two-loop recursion \citep{nocedal1980updating}. Hence, one can conclude that the computational cost and storage requirement of BFGS is $\ccalO(\min(td,d^2))$.

The above observation motivates the use of limited-memory QN methods and in particular limited memory BFGS (L-BFGS), in which we only keep last $\tau$ pairs of curvature information, i.e., $\{(\bbs_u, \bbr_u)\}_{u=t-\tau}^{t-1}$, to construct the descent direction at each iteration. If $\tau <t $ and $\tau<d$ it will reduce the computational cost and storage requirement of BFGS from $\ccalO(\min(td,d^2))$ to $\ccalO(\tau d)$. Indeed, a larger $\tau$ leads to a better approximation of the BFGS descent direction but higher computation and memory, while a smaller $\tau$ reduces the memory and computation at the price of having a less accurate approximation of the BFGS direction. 
A key observation is that for such limited memory updates the standard analysis of the Hessian approximation error does not go through, and hence the theoretical results for standard QN methods (without limited memory) may not hold anymore. To the best of our knowledge, all existing analyses of L-BFGS can only provide linear convergence rates, which is not better than the one for gradient descent \citep{liu1989limited,mokhtari2015global,moritz2016linearly,bollapragada2018progressive,berahas2020robust,berahas2022quasi}. Therefore, the question is how to design a memory-retention procedure so that a non-asymptotic superlinear rate can be achieved in a limited-memory manner. 
Doing so is the focus of the following section.

\section{Limited-Memory Greedy BFGS (LG-BFGS)} 

We propose a new variant of L-BFGS that departs from the classical approach in two key ways: (i)~greedy basis vector selection for the variable variation and (ii)~displacement aggregation on the gradient variation. Greedy updates have been introduced into QN methods as a way to maximize a measure of progress, the trace Lyapunov function, w.r.t. variable variation, which enables to derive an explicit per-step improvement of the Hessian approximation to the true Hessian. The upshot of this modification is an explicit (non-asymptotic) superlinear convergence rate~\citep{rodomanov2021greedy}. Moreover, as we will see, the basis vector selection of the greedy update allows to control the memory growth, which is not present in classical QN schemes and provides a quantitative motivation to port greedy updates \citep{rodomanov2021greedy} into the limited-memory setting.

However, greedy selection alone with respect to variable variation cannot ensure that the resultant L-BFGS retains the correct curvature information, as L-BFGS stores only a limited number of past gradient values. Displacement aggregation \citep{berahas2022limited} provides a way to cull linearly dependent curvature information such that a retained subset gradient variation information is sufficient for ensuring the quality of the Hessian inverse approximation, which can result in asymptotic superlinear rates \citep{berahas2022limited, sahu2023modified}. We synthesize these two key strategies in a unique fashion and present our Limited-memory Greedy BFGS (LG-BFGS) method. 

Let $\bbx_0$ be the initial variable, $\nabla f(\bbx_0)$ the initial gradient, 
and $\bbB_0$ the initial Hessian approximation satisfying $\bbB_0 \succeq \nabla^2\! f(\bbx_0)$, which can be easily satisfied by setting $\bbB_0=L \bbI$.
The LG-BFGS consists of two phases at each iteration: (i) decision variable update and (ii) curvature pair update.

\smallskip
\noindent \textbf{Decision variable update:} At iteration $t$, we have access to the initial Hessian approximation, previously stored curvature pairs, the number of which is at most $\tau$, the current iterate $\bbx_t$ and its gradient $\nabla f(\bbx_t)$. Similar to L-BFGS, we can compute the descent direction of LG-BFGS denoted by $\bbd_t$ at a cost of at most $\mathcal{O}(\tau d)$ and then find the new iterate with stepsize $\alpha$ as 
%
\begin{equation}\label{eq:LGBFGSvariable}
	\bbx_{t+1} = \bbx_t +\alpha \bbd_t.
\end{equation}
%

\smallskip
\noindent \textbf{Curvature pair update:} The update of the curvature pairs happens in two steps as described below.

\smallskip
\noindent \emph{(1) Greedy step.} The greedy step selects the curvature pair greedily to maximize the update progress of the Hessian approximation. Specifically, it selects the variable variation $\bbs_t$ from a vector basis $\{ \bbe_1,\ldots,\bbe_d \}$ by maximizing the ratio $\frac{\bbs^\top \bbB_t\bbs}{\bbs^\top \nabla^2 f(\bbx_{t+1})\bbs}$, where 
$\{ \bbe_1,\ldots,\bbe_d \}$ can be any basis in the space.  For ease of implementation, a reasonable choice is selecting $\bbe_i$ as the vector that its $i$-th entry is the only non-zero value. 
For the purpose of limited memory, we propose an innovative search scheme that 
restricts the vector basis to the subset $\{ \bbe_1,\ldots,\bbe_\tau \}$ of $\tau < d$ vectors, and select $\bbs_t$ greedily as 
\begin{equation}\label{eq:greedyPreparation}
		\bbs_t = \argmax_{\bbs \in \{ \bbe_1,\ldots,\bbe_\tau \}} \frac{\bbs^\top \bbB_t\bbs}{\bbs^\top \nabla^2 f(\bbx_{t+1})\bbs}.
\end{equation}
The greedy step 
selects $\bbs_t$ such that it maximizes the progress of Hessian approximation. For the creation of our curvature pair, instead of using the classic $\bbx_{t+1} - \bbx_t$ as the variable variation, we use the output $\bbs_t$ of \eqref{eq:greedyPreparation}. Note that 
\eqref{eq:greedyPreparation} is equivalent to computing $\tau$ diagonal components of the matrix $\bbB_t$ and $\nabla^2 f(\bbx_{t+1})$, which can be done efficiently at a cost of $\ccalO(\tau d)$, and matrix-vector products $\{\bbB_t \bbe_i\}_{i=1}^\tau$ can be computed with a compact representation at a cost of $\ccalO(\tau^2 d)$ without forming $\bbB_t$; see Section 7.2 in \citep{nocedal2006numerical}. This is due to the fact that $\bbB_t$ only depends on 
$\tau$ curvature pairs. Once $\bbs_t$ is computed, we can efficiently compute $\bbr_t$ by following 
$\bbr_t= \nabla^2 f(\bbx_{t+1}) \bbs_t$. Since $\bbs_t$ is a unit vector with only one non-zero entry, we can simply compute $\bbr_t$ by adding the elements of one column of $\nabla^2 f(\bbx_{t+1}) $. Hence, all computations are $\mathcal{O}(d)$.

\smallskip
\noindent \emph{(2) Displacement step.} Given access to $\hat{\tau}$ previously stored curvature pairs $\ccalP_{t-1} = \{(\bbs_{u}, \bbr_{u})\}_{u=0}^{\hat{\tau}-1}$, where $\hat{\tau}$ is either less than or equal to our memory budget, i.e., $\hat{\tau} \le \tau$, and a new curvature pair $(\bbs_t, \bbr_t)$ from the evaluation of \eqref{eq:greedyPreparation}, the key question is how to select which information to retain for curvature pair update. Displacement aggregation incorporates $(\bbs_t, \bbr_t)$ into $\ccalP_{t-1}$ and formulates an aggregated curvature set $\ccalP_t$ rather than simply removing the oldest pair $(\bbs_0, \bbr_0)$ and adding the latest one $(\bbs_t, \bbr_t)$. Doing so yields a limited-memory Hessian approximation equal to that computed from a full-memory version; hence, reducing the effect of memory reduction and improving the convergence rate. The execution of this retention strategy contains three cases:
 
\smallskip
	{\bf (C1)} {\bf If new variable variation $\bbs_t$ is linearly independent of $\{\bbs_{u}\}_{u=0}^{\hat{\tau}-1}$ in $\ccalP_{t-1}$}, 
 we formulate $\ccalP_t$ by directly adding the new curvature pair as 
	\begin{align} \label{eq:displacementCase1}
		\ccalP_t = \{\ccalP_{t-1}, (\bbs_t,\bbr_t)\}.
	\end{align}	
	
\smallskip
	{\bf (C2)} {\bf If new variable variation $\bbs_t$ is parallel to last variable variation $\bbs_{\hat{\tau}-1}$ in $\ccalP_{t-1}$}, 
    we formulate $\ccalP_t$ by removing the most recently stored 
    $(\bbs_{\hat{\tau}-1}, \bbr_{\hat{\tau}-1})$ and replace it with the new 
    $(\bbs_t,\bbr_t)$ as 
	\begin{align}\label{eq:displacementCase2}
		\ccalP_t = \{(\bbs_{0},\bbr_{0}), \ldots, (\bbs_{{\hat{\tau}-2}}, \bbr_{{\hat{\tau}-2}}), (\bbs_t,\bbr_t)\}.
	\end{align}
    I.e., we have $(\bbs_{\hat{\tau}-1}, \bbr_{\hat{\tau}-1}) = (\bbs_t,\bbr_t)$ in the updated $\ccalP_t$. 
	
\smallskip
	{\bf (C3)} {\bf If new variable variation $\bbs_t$ is parallel to a previously stored variable variation in $\ccalP_{t-1}$}, 
    we formulate $\ccalP_t$ by projecting the information of $(\bbs_{t}, \bbr_{t})$ onto the previous pairs to modify and reweigh $\ccalP_{t-1}$. Specifically, assume $\bbs_t = \bbs_{j}$ for some $0 \le j < \hat{\tau}-1$. Denote by $\bbS_{{j_1}:{j_2}} = [\bbs_{{j_1}} \cdots \bbs_{{j_2}}] \in \mathbb{R}^{d \times (j_2-j_1+1)}$ and $\bbR_{{j_1}:{j_2}}= [\bbr_{{j_1}} \cdots \bbr_{{j_2}}] \in \mathbb{R}^{d \times (j_2-j_1+1)}$ the concatenated variable variation matrix and gradient variation matrix for $0 \le j_1 \le j_2$, and define $(\bbs_{{\hat{\tau}}}, \bbr_{{\hat{\tau}}}) = (\bbs_t, \bbr_t)$ as the new curvature pair. We remove $(\bbs_{j}, \bbr_{j})$, replace $\bbR_{{j+1}:{\hat{\tau}}}$ with a modified $\hat{\bbR}_{{j+1}:{\hat{\tau}}}$, and formulate $\ccalP_t$ as
	\begin{align}\label{eq:displacementCase3}
		\ccalP_t= \{(\bbs_{0}, \bbr_{0}),\ldots,(\bbs_{{j-1}}, \bbr_{{j-1}}), (\bbs_{{j+1}}, \hat{\bbr}_{{j+1}}),\ldots, (\bbs_{{\hat{\tau}}}, \hat{\bbr}_{{\hat{\tau}}})\} 
	\end{align}
	such that the descent direction in \eqref{eq:LGBFGSvariable} computed from $\ccalP_t$ is equal to that from $\{\ccalP_{t-1}, (\bbs_{\hat{\tau}},\bbr_{\hat{\tau}})\}$. To do so, we compute $\hat{\bbR}_{{j+1}:{\hat{\tau}}} \!=\! [\hat{\bbr}_{{j+1}} \cdots \hat{\bbr}_{{\hat{\tau}}}]$ following the aggregation procedure in \citep{berahas2022limited} as
	\begin{align}\label{eq:state_aggregation_projection}
		\hat{\bbR}_{{j+1}:{\hat{\tau}}} = (\bbH_{0:{j-1}})^{-1} \bbS_{{j+1}:{\hat{\tau}}}[\bbA~\textbf{0}] + \bbr_{j} [\bbb~\bb0] + \bbR_{{j+1}:{\hat{\tau}}}
	\end{align}
	where $\bbH_{0:{j-1}}$ is the Hessian inverse approximation computed with the limited-memory strategy, initialized at $\bbH_0$ and updated with $j$ curvature pairs $\{(\bbs_{u}, \bbr_{u})\}_{u=0}^{j-1}$, and $\bbb \in \mathbb{R}^{\hat{\tau}-j-1}$ and $\bbA \in \mathbb{R}^{(\hat{\tau}-j)\times(\hat{\tau}-j-1)}$ are unknowns that can be computed by Algorithm 4 in \citep{berahas2022limited}. Note that $\bbH_{0:{j-1}}^{-1} \bbS_{{j+1}:{\hat{\tau}}}$ in (\ref{eq:state_aggregation_projection}) can be computed without forming $\bbH_{0:{j-1}}$ \citep{berahas2022limited}.

\smallskip
We remark that the number of curvature pairs $\hat{\tau}$ in $\ccalP_t$ is always bounded by the memory size $\tau$, i.e., $\hat{\tau} \le \tau$ for all iterations $t$. We explain this fact from two aspects:

\smallskip
\noindent \emph{(i)} (C1) only happens when $\hat{\tau} < \tau$. Specifically, the variable variations $\{\bbs_{u}\}_{u=0}^{\hat{\tau}-1}$ in $\ccalP_{t-1}$ are independent 
because any dependent variable variation will not be included into $\ccalP_{t-1}$ according to (C2) and (C3). If $\hat{\tau} \ge \tau$, the new variable variation $\bbs_t$ cannot be linearly independent of $\{\bbs_{u}\}_{u=0}^{\hat{\tau}-1}$ in $\ccalP_{t-1}$ because all variable variations are selected from the subset $\{ \bbe_1,\ldots,\bbe_\tau \}$ of size $\tau$ [cf. \eqref{eq:greedyPreparation}]. Thus, the number of curvature pairs in $\ccalP_t$ is bounded by the memory size, i.e.,  $\hat{\tau}+1 \le \tau$. 

\smallskip
\noindent \emph{(ii)} (C2) and (C3) do not increase the number of curvature pairs from $\ccalP_{t-1}$ to $\ccalP_t$ and thus, the number of curvature pairs in $\ccalP_t$ is bounded by the memory size, i.e., $\hat{\tau} \le \tau$. 

\smallskip
\noindent Therefore, LG-BFGS stores at most $\tau$ curvature pairs 
and solves problem \eqref{eq:ERMproblem1} with limited memory. 

\begin{remark}[Computational Cost]
While the displacement step seems computationally costly, it can be implemented efficiently. 
For the case selection, it is computationally efficient because variable variations $\{\bbs_{0},\ldots, \bbs_{{\hat{\tau}-1}}\}$ are selected from the same subset $\{\bbe_1,\ldots,\bbe_\tau\}$ instead of distributed arbitrarily in  $d$-dimensional space and thus, the computational cost is of $\ccalO(\tau)$. For the modified $\hat{\bbR}_{{j+1}:\hat{\tau}}$, the computational cost is of $\ccalO(\tau^2 d + \tau^4)$ as analyzed in \citep{berahas2022limited}. Hence, the overall cost is comparable to the cost of L-BFGS which is  $\ccalO(4\tau d)$ \citep{liu1989limited}. 
\end{remark} 

\begin{algorithm}[t!]
	\caption{Limited-memory Greedy BFGS (LG-BFGS) Method}\label{algorithm:LG-BFGS}
	\textbf{Initialization:} {Loss function $f(\bbx)$, initial decision vector $\bbx_0$ and Hessian inverse approximation $\bbH_0$}\\
	\textbf{for} {$t = 0,1,\ldots,T$}~~\textbf{do}\\
		1. Compute the gradient $\nabla f(\bbx_{t})$ and update the decision variable $\bbx_{t+1}$ as in \eqref{eq:LGBFGSvariable}\\
		2. Compute $\phi_t \!=\! \| \bbx_{t+1}\!-\!\bbx_t\|_{\nabla^2\! f(\bbx_t)}$, the scale constant $\psi_t = 1+C_M \phi_t$, the corrected initial $\bbH_0 = \psi_{t}^{-1} \bbH_{0}$ and the scaled curvature pairs $\{\bbs_u, \tilde{\bbr}_u\}_{u=0}^{\hat{\tau}-1}$ with $ \tilde{\bbr}_u = \psi_{t}\bbr_u$\\
		3. Select the greedy curvature pair $(\bbs_t, \bbr_t)$ as in \eqref{eq:greedyPreparation} with the corrected curvature pairs $\{\bbs_u, \tilde{\bbr}_u\}_{u=0}^{\hat{\tau}-1}$\\
		4. Update the historical curvature pairs $\ccalP_{t}$ as in \eqref{eq:displacementCase1}-\eqref{eq:displacementCase3}
\end{algorithm}

\subsection{Correction Strategy}\label{sec:correction}

The greedy step 
is designed to maximize 
the progress of the BFGS update in terms of the trace function $\sigma(\nabla^2 f(\bbx_{t+1}), \bbB_{t+1}\!) \!=\! \text{Tr}(\nabla^2 f(\bbx_{t+1})^{-1} \bbB_{t+1} \!)- d$, which captures the difference between positive definite matrices. For our analysis, we require this expression to stay positive and for that reason, we need the following condition ${\bbB}_{t+1} \succeq \nabla^2 f(\bbx_{t+1})$. Note that this condition may not hold even if $\bbB_t \succeq \nabla^2 f(\bbx_{t})$, which requires developing a correction strategy to overcome this issue. Specifically, define $\phi_t \!:=\! \| \bbx_{t+1}-\bbx_{t}\|_{\nabla^2\! f(\bbx_t)} $ and 
the corrected matrix $\hat{\bbB}_t \!=\! (1\!+\! \phi_t C_M)\bbB_t \succeq \nabla^2 f(\bbx_{t+1})$ as a scaled version of $\bbB_t$ where $C_M$ is the strongly self-concordant constant; see \eqref{eq:SelfConcordant}. By updating $\bbB_{t+1}$ with the corrected $\hat{\bbB}_t$, we can ensure that $\bbB_{t+1}\succeq \nabla^2 f(\bbx_{t+1})$ is satisfied \citep{rodomanov2021greedy}. However, in the limited memory setting, we do not have access to the Hessian approximation matrix explicitly and thus need to apply this scaling to the curvature pairs. This can be done 
by scaling the gradient variation as $\tilde{\bbr}_t= (1 + \phi_t C_M) \bbr_t$. In fact, this means that we need to use the corrected gradient variation $\tilde{\bbr}_t$ instead of $\bbr_t$ for the displacement step. 

Next, we formally prove that using the scaled gradient variation ${\tilde{\bbr}_t}$ would equivalently lead to scaling the corresponding Hessian approximation matrices. 

\begin{proposition}\label{prop:correctStep}
	Consider the BFGS update in \eqref{eq:BFGS}. Let $\bbB_0$ be the initial Hessian approximation, $\psi_{t} = 1+ \phi_{t} C_M$ be the scaling parameter, and $\{(\bbs_u,\bbr_u)\}_{u=0}^{t-1}$ be the curvature pairs. Define the modified curvature pairs $\{(\bbs_u,\tilde{\bbr}_u)\}_{u=0}^{t-1}$ as $\tilde{\bbr}_u = \psi_{t} \bbr_u$ for all $u=0,\ldots,t-1$. At iteration $t$, let $\hat{\bbB}_t$ be the corrected Hessian approximation as
	\begin{align}\label{eq:correctedBFGS}
		\hat{\bbB}_t = \psi_{t} \text{\rm BFGS}(\bbB_{t-1}, \bbs_{t-1}, \bbr_{t-1}),~\ldots, ~\bbB_1 = \text{\rm BFGS}(\bbB_{0}, \bbs_{0}, \bbr_{0})
	\end{align}
 where $\text{\rm BFGS}(\cdot)$ represents the operation in \eqref{eq:BFGS}. Then, given $\tilde{\bbB}_{0} = \psi_{t} \bbB_0$, it holds that 
	\begin{align}\label{eq:alterCorrectedBFGS}
		\hat{\bbB}_t = \text{\rm BFGS}(\tilde{\bbB}_{t-1}, \bbs_{t-1}, \tilde{\bbr}_{t-1}),~\ldots, ~\tilde{\bbB}_1 = \text{\rm BFGS}(\tilde{\bbB}_{0}, \bbs_{0}, \tilde{\bbr}_{0}).
	\end{align}
\end{proposition}
\noindent Proposition \ref{prop:correctStep} states that we can transform the correction strategy of the Hessian approximation $\bbB_t$ [cf. \eqref{eq:correctedBFGS}] to the corresponding correction strategy of the gradient variation $\bbr_t$. This indicates that we can incorporate the correction strategy into the displacement step by scaling the gradient variation $\bbr_t$ and maintain the remaining unchanged. Algorithm \ref{algorithm:LG-BFGS} formally summarizes the corrected LG-BFGS method which alleviates the requirement of $\bbB_{t+1} \succeq \nabla^2 f(\bbx_{t+1})$.

\section{Convergence}\label{sec:convergence}

In this section, we conduct convergence analysis to show that LG-BFGS can achieve an explicit superlinear rate depending on the memory size $\tau$. 
This result is salient as it is the first non-asymptotic superlinear rate established by a limited-memory quasi-Newton method. Moreover, our result identifies how memory size affects convergence speed, which provides an explicit trade-off between these two factors. 
To proceed, we introduce some technical conditions on the objective function $f$. 
%
\begin{assumption} \label{Assump:strongConvex}
	The function $f$ is $\mu$-strongly convex and its gradient is $L$-Lipschitz continuous. 
\end{assumption}
\begin{assumption}\label{Assump:LipschitzHessian}
	The Hessian $\nabla^2 f$ is $C_L$-Lipschitz continuous, i.e.,  for all $\bbx_1,\bbx_2 \in \mathbb{R}^d$, we have $\| \nabla^2 f(\bbx_1) - \nabla^2 f(\bbx_2)\| \le C_L \| \bbx_1 - \bbx_2 \|$.
\end{assumption}
\noindent Assumptions \ref{Assump:strongConvex} and  \ref{Assump:LipschitzHessian} are 
customary in the analysis of QN methods \citep{nocedal2006numerical}[Ch. 6]. They also together imply that  $f$ is strongly self-concordant with constant $C_M = C_L/\mu^{3/2}$, i.e., 
\begin{equation}\label{eq:SelfConcordant}
		\begin{split}
			\nabla^2 f(\bbx_1) - \nabla^2 f(\bbx_2) \preceq C_M \| \bbx_1 - \bbx_2 \|_{\nabla^2\! f(\bby)} \nabla^2 f(\bbz),~\text{for~any}~ \bbx_1,\bbx_2,\bby,\bbz \in \mathbb{R}^d.
		\end{split}
\end{equation}%
%
With these preliminaries in place, we formally present the convergence results. Specifically, we consider the convergence criterion as the operator norm that has been commonly used in \citep{rodomanov2021greedy,lin2021greedy,rodomanov2021rates,jin2022non}
\begin{align}
    \lambda_f(\bbx_t) := \|\nabla f(\bbx_t)\|_{\nabla^2 f(\bbx_t)^{-1}}
\end{align}
and show the sequence $\{\lambda_f(\bbx_t)\}_{t=1}^T$ generated by LG-BFGS can converge to zero at an explicit superlinear rate. To do so, we first state the linear convergence of $\{\lambda_f(\bbx_t)\}_{t=1}^T$ for LG-BFGS. 

\begin{proposition}\label{prop:linearConvergence}
	Consider LG-BFGS for problem \eqref{eq:ERMproblem1} satisfying Assumptions \ref{Assump:strongConvex}-\ref{Assump:LipschitzHessian} with $\mu$, $L$ and $C_M$. Let $\epsilon(\mu, L, C_M) =  \mu \ln (3/2) / (4 L C_M)$. Then, if $\bbx_0$ satisfies $\lambda_f(\bbx_0) \le \epsilon(\mu, L, C_M)$, we have
	\begin{align}\label{eq:thmlinear}
		\lambda_f(\bbx_t) \le \left(1 - \frac{\mu}{2 L}\right)^t \lambda_f(\bbx_0).
	\end{align}
\end{proposition}
\noindent Proposition \ref{prop:linearConvergence} states that the weighted gradient norm $\lambda_f(\bbx_t)$ of the iterates of LG-BFGS converges to zero at least at a linear rate. This is a local linear convergence result, where the initial variable $\bbx_0$ is assumed close to the solution, i.e., $\lambda_f(\bbx_0) \le \epsilon(\mu, L, C_M)$. 
Before proceeding with our superlinear analysis, we define the \textit{relative condition number} of basis $\{\bbe_i\}_{i=1}^d$ with respect to any matrix $\bbE$.

%
\begin{definition}\label{def:relativeCondition}[Relative Condition Number]
	Consider a matrix $\bbE$ and a vector basis $\{\bbe_i\}_{i=1}^d$. The relative condition number of each vector $\bbe_i$ with respect to $\bbE$ is defined as
\begin{equation}\label{eq:relativeCondition}
		\beta(\bbe_i) := \frac{\max_{1 \le k \le d} \bbe_k^\top \bbE \bbe_k}{ \bbe_i^\top \bbE \bbe_i}.
	\end{equation}
Moreover, we define the minimal relative condition number of the subset $\{\bbe_i\}_{i=1}^\tau$ as 
 $$
 \beta_\tau :=\beta(\{\bbe_i\}_{i=1}^\tau) = \min_{1 \le i \le \tau} \beta(\bbe_i).
 $$
\end{definition}
The relative condition number in Definition \ref{def:relativeCondition} characterizes how large the linear operation $\bbE \bbx$ can change for a small change on the input $\bbx$ along the direction $\bbe_i$ for $i=1,\ldots,d$. Indeed, based on the definition in \eqref{eq:relativeCondition}, this value is larger than or equal to $1$. It can be also verified that the maximum possible value for the relative condition number is $\lambda_d/\lambda_1$, where $\lambda_d$ and $\lambda_1$ are the maximal and minimal eigenvalues of $\bbE$ (we assume $\lambda_1\le \cdots \le \lambda_d$). If particularizing 
$\{\bbe_i\}_{i=1}^d$ to the eigenvectors $\{\bbv_i\}_{i=1}^d$ of $\bbE$, i.e., $\{\bbe_i = \bbv_i\}_{i=1}^d$, the relative condition number of $\bbe_1$ with respect to $\bbE$ is equivalent to the condition number of $\bbE$, i.e., $\lambda_d/\lambda_1$, and the relative condition number of $\bbe_d$ is 1. 

Moreover, $ \beta(\{\bbe_i\}_{i=1}^\tau)$, abbreviated as $\beta_\tau$, measures the minimum possible relative condition number for the elements in the subset $\{\bbe_i\}_{i=1}^\tau$. 
When increasing the size of the set, i.e., increasing $\tau$, the minimal relative condition number decreases and approaches 1. In fact, for $\tau=d$ we obtain $\beta_d=1$. We next use this definition to characterize the progress of Hessian approximation in LG-BFGS. We follow \citep{rodomanov2021greedy} to measure the Hessian approximation error with the trace metric $\sigma(\nabla^2 f(\bbx), \bbB) \!=\! {\rm Tr}(\nabla^2 f(\bbx)^{-1} \bbB) \!-\! d$, which captures the distance between $\nabla^2 f(\bbx)$ and $\bbB$.
\begin{proposition}\label{prop:HessianUpdate}
    Suppose Assumptions \ref{Assump:strongConvex}-\ref{Assump:LipschitzHessian} hold and  $\nabla^2 f(\bbx) \preceq \bbB$. Let $\bbx_+$ be the updated decision variable by \eqref{eq:LGBFGSvariable}, $\phi = \| \bbx_+ - \bbx \|_{\nabla^2 f(\bbx)}$ the weighted update difference, $\hat{\bbB} = (1 + \phi C_M) \bbB$ the corrected Hessian approximation, $(\bbs, \bbr)$ the greedily selected curvature pair by \eqref{eq:greedyPreparation}, $\bbB_+ = \text{\rm BFGS}(\hat{\bbB}, \bbs, \bbr)$ the updated Hessian approximation, and $\tau$ the memory size. Then, we have
\begin{equation}\label{eq:propHessianUpdate}
		\sigma(\nabla^2 f(\bbx_+), \bbB_+)
        \le \bigg(1 - \frac{\mu}{\beta_\tau d L}\bigg)(1 + \phi C_M)^2 \Big(\sigma(\nabla^2 f(\bbx), \bbB) + \frac{2 d \phi C_M}{1+ \phi C_M}\Big).
	\end{equation} 
\end{proposition}
\noindent Proposition \ref{prop:HessianUpdate} captures the contraction factor for the convergence of the Hessian approximation error in LG-BFGS, which is determined by not only the objective function but also the memory size. Specifically, the term $1 - \mu/(\beta_\tau d L)$ in \eqref{eq:propHessianUpdate} embeds the role played by the memory size $\tau$. Since $\beta_{\tau_1} \ge \beta_{\tau_2}$ for any $\tau_1 \le \tau_2$ by Definition~\ref{def:relativeCondition}, a larger $\tau$ allows the greedy selection from a larger subset $\{\bbe_i\}_{i=1}^\tau$, yields a smaller minimal relative condition number $\beta_\tau$, and generates a smaller contraction factor that improves the update progress. While the update progress is shown in \eqref{eq:propHessianUpdate}, we remark that LG-BFGS does not compute any Hessian approximation matrix $\bbB$ across iterations. By combining Propositions \ref{prop:linearConvergence}-\ref{prop:HessianUpdate}, we formally show 
the convergence of LG-BFGS. 
\begin{theorem}\label{thm:superlinearConvergence}
	Consider the setting in Proposition \ref{prop:linearConvergence}. Let $\beta_{t, \tau}$ be the minimal relative condition number of the subset $\{\bbe_u\}_{u=1}^\tau$ with respect to the error matrix $\hat{\bbB}_t - \nabla^2 f(\bbx_{t+1})$ at iteration $t$, and $t_0$ be such that $(2 d L / \mu) \prod_{u=1}^{t_0}(1\!-\!\frac{\mu}{\beta_{u,\tau}dL}) \le 1$. Then, if $\bbx_0$ satisfies $\lambda_f(\bbx_0) \le  \mu \ln 2 / (4 (2d+1) L)$, we have
\begin{equation}\label{eq:ExplicitSuperLinearRate}
		\lambda_f(\bbx_{t+t_0+1}) \le \prod_{u=t_0+1}^{t+t_0}\Big(1\!-\!\frac{\mu}{\beta_{u,\tau}dL}\Big)^{t+t_0+1-u} 
  \Big(1\!-\!\frac{\mu}{2L}\Big)^{t_0} \lambda_f(\bbx_{0}).
	\end{equation}
\end{theorem}
Theorem \ref{thm:superlinearConvergence} states that the iterates of LG-BFGS converge to the optimal solution and establishes an explicit convergence rate, which depends on the linear term $(1-\mu/(2L))^{t_0}$ and the quadratic term $\prod_{u=t_0+1}^{t+t_0}(1\!-\!\mu/(\beta_{u,\tau}dL))^{t+t_0+1-u}$. By introducing a structural assumption on the error matrix $\hat{\bbB}_t - \nabla^2 f(\bbx_{t+1})$, we can simplifies the expression of the convergence rate and show a non-asymptotic superlinear rate in the following theorem.
\begin{theorem}\label{coro:superlinearConvergence}
	Under the same settings as Theorem \ref{thm:superlinearConvergence}, suppose the minimal relative condition number $\beta_{t,\tau}$ of subset $\{\bbe_i\}_{i=1}^\tau$ 
 or the condition number $\beta_{t}$ of error matrix $\hat{\bbB}_t - \nabla^2 f(\bbx_{t+1})$ is bounded by a constant $C_\beta$, i.e., either $\beta_{t,\tau} \le C_\beta$ or $\beta_t \le C_\beta$. 
 Then, we have \begin{align}\label{eq:ExplicitSuperLinearRateCoro}
		\lambda_f(\bbx_{t+t_0+1}) \le \Big(1\!-\!\frac{\mu}{C_\beta dL}\Big)^{\frac{t(t+1)}{2}} 
        \Big(1\!-\!\frac{\mu}{2L}\Big)^{t_0} \lambda_f(\bbx_{0}).
	\end{align}
\end{theorem}
Theorem \ref{coro:superlinearConvergence} presents concrete takeaways demonstrating the explicit superlinear rate of LG-BFGS. For the first $t_0$ iterations, LG-BFGS converges at a linear rate that is not affected by the memory limitation. For $t \ge t_0$, once the superlinear phase is triggered, LG-BFGS converges at a superlinear rate whose contraction factor depends on the condition number $C_\beta$ of the error matrix of Hessian approximation, which in turn depends on the memory size $\tau$. 
This is the first result showing a superlinear rate for a fixed-size limited memory QN method, while it is important to note that it only holds for a sub-class of problems where the relative condition number of the Hessian approximation error along a low-memory subspace is well-behaved. We discuss more details 
in Remark \ref{remark:memory}.

\smallskip
\noindent {\bf Effect of memory size.} The fact that we only employ at most $\tau$ curvature pairs affects the convergence rate in a way previously unaddressed / unclear in the literature. Specifically, a larger $\tau$ decreases the minimal relative condition number $\beta_{u,\tau}$, reduces the contraction factor $1-\mu/(\beta_{u,\tau}dL)$, and improves the convergence rate, but increases the storage memory and computational cost; hence, yielding an inevitable trade-off between these factors. 
Importantly, this result  recovers the explicit superlinear rate of greedy BFGS \citep{rodomanov2021greedy} when the subset $\{\bbe_u\}_{u=1}^\tau$ increases to the entire basis $\{\bbe_u\}_{u=1}^d$, i.e., the limited memory increases to the full memory. 
\begin{remark}\label{remark:memory}
    There may exist pessimistic scenarios, where the 
    assumption on the error matrix $\hat{\bbB}_t - \nabla^2 f(\bbx_{t+1})$ does not hold for a small memory size $\tau$, and the convergence rate of LG-BFGS in Theorem \ref{thm:superlinearConvergence} may not be superlinear as shown in Theorem \ref{coro:superlinearConvergence}. However, 
    we show in Appendix F that there exists a provable bound $C_{t, \beta}$ on the condition number $\beta_t$ of $\hat{\bbB}_t - \nabla^2 f(\bbx_{t+1})$ in any circumstances, which increases with the iteration $t$, and LG-BFGS will converge at least with an improved linear rate. We emphasize that this bound is the worst-case analysis established on the condition number $\beta_t$, i.e., the minimal relative condition number with memory size one $\beta_{t,1}$. 
\end{remark}

\section{Discussion}\label{sec:discussion}

In this section, we compare the convergence rates, as well as storage requirements of LG-BFGS with other quasi-Newton methods. We replace all universal constants with $1$ to ease the comparisons.

\vspace{1mm}
\noindent \textbf{LG-BFGS.} From 
Theorem~\ref{coro:superlinearConvergence}, the iterates of 
LG-BFGS satisfy 
$
\frac{\lambda_f(\bbx_t)}{\lambda_f(\bbx_0)} \le \min \{ (1-\frac{\mu}{L})^t, (1-\frac{\mu}{C_\beta dL})^{\frac{t(t+1)}{2}} \}$.
%
When $t < C_\beta dL \ln(dL/\mu) / \mu$, the superlinear phase is not yet triggered and the first term is smaller which implies a linear rate of $(1-\mu/L)^t$. Once the superlinear phase is triggered, i.e., $t \ge C_\beta dL \ln(dL/\mu) / \mu$, the second term becomes smaller and iterates converge at a superlinear rate of $(1-\mu/(C_\beta dL))^{t(t+1)/2}$. The constant $C_{\beta}$ (defined in Theorem \ref{coro:superlinearConvergence}) depends on the memory size $\tau$, and as $\tau$ increases, $C_{\beta}$ becomes smaller and the superlinear rate becomes faster. 
The memory storage requirement of LG-BFGS is $\ccalO(\tau d)$ and its per iteration complexity is $\ccalO(\tau^2 d+\tau^4)$.

\vspace{1mm}
\noindent \textbf{L-BFGS.} The iterates of L-BFGS converge at a linear rate \citep{liu1989limited}, i.e., 
$|f(\bbx_t) - f(\bbx^*)| \le \gamma^t |f(\bbx_0) - f(\bbx^*)|$, for some $ \gamma \in (0,1)$, which is slower than the superlinear rare of LG-BFGS. On the other hand, the storage requirement and  cost per iteration of L-BFGS are of $\ccalO(\tau d)$, which are comparable with the ones for LG-BFGS.

\vspace{1mm}
\noindent \textbf{Greedy BFGS.} Based on \citep{rodomanov2021greedy,lin2021greedy}, the iterates of Greedy BFGS satisfy 
$\frac{\lambda_f(\bbx_t)}{\lambda_f(\bbx_0)} \le \min \{ (1-\frac{\mu}{L})^t, (1-\frac{\mu}{dL})^{\frac{t(t+1)}{2}} \}$. It 
requires $dL \ln(dL/\mu) / \mu$ iterations to trigger the superlinear phase and 
achieves a faster superlinear rate than LG-BFGS since $C_\beta \ge 1$. However, LG-BFGS requires less storage and computational cost per iteration compared to greedy BFGS, since the storage and per iteration complexity of greedy BFGS are $\ccalO( d^2)$.

\vspace{1mm}
\noindent \textbf{BFGS.} The iterates of BFGS 
satisfy 
$ \frac{\lambda_f(\bbx_t)}{\lambda_f(\bbx_0)} \le \min \{ (1-\frac{\mu}{L})^t, ( \frac{d \ln ({L}/{\mu})}{t} )^{\frac{t}{2}} \}$ \citep{rodomanov2021rates}.
The superlinear convergence of BFGS starts after $d \ln (L/\mu)$ iterations, while it takes $C_\beta dL \ln(dL/\mu) / \mu$ iterations for LG-BFGS to trigger the superlinear rate. However, 
the superlinear rate of LG-BFGS is faster than BFGS for large $t$, i.e.,
$    (1\!-\!\frac{\mu}{C_\beta dL})^{\frac{t(t+1)}{2}} \ll ( \frac{(d \ln \frac{L}{\mu})}{t})^{\frac{t}{2}}$.
Moreover, the storage requirement of LG-BFGS and its cost per iteration are smaller than $\ccalO( d^2)$ of BFGS. 

\vspace{1mm}
\noindent \textbf{BFGS with displacement aggregation.} The results \citep{sahu2023modified,berahas2022limited} for BFGS with displacement aggregation can only guarantee an \textit{asymptotic} superlinear convergence. Moreover, it is not strictly limited-memory as it computes the variable variation as $\bbs_t = \bbx_{t+1} - \bbx_t$ which can be any vector in $\reals^d$. Hence, $\bbs_t$ could be independent of all previously stored variable variations, which leads to (C1) in the displacement step and increases the memory size [cf. \eqref{eq:displacementCase1}]. As a result, 
the memory size could increase up to $d$, resulting in a storage and cost per iteration of $\ccalO(d^2)$.  
In contrast, LG-BFGS selects $\bbs_t$ from  $ \{\bbe_i\}_{i=1}^\tau$ and the maximal number of independent variable variations is $\tau$. Hence, $\bbs_t$ is guaranteed 
parallel to one of the previously stored variable variations when the memory size reaches $\tau$ and LG-BFGS is strictly limited-memory with the storage requirement of $\ccalO(\tau d)$. Moreover, LG-BFGS obtains a non-asymptotic superlinear convergence rate, which explains the trade-off between memory and convergence rate explicitly.

\section{Experiments}

\begin{figure}[t]
  \centering
  \begin{subfigure}{0.42\linewidth}
    \includegraphics[width=\linewidth]{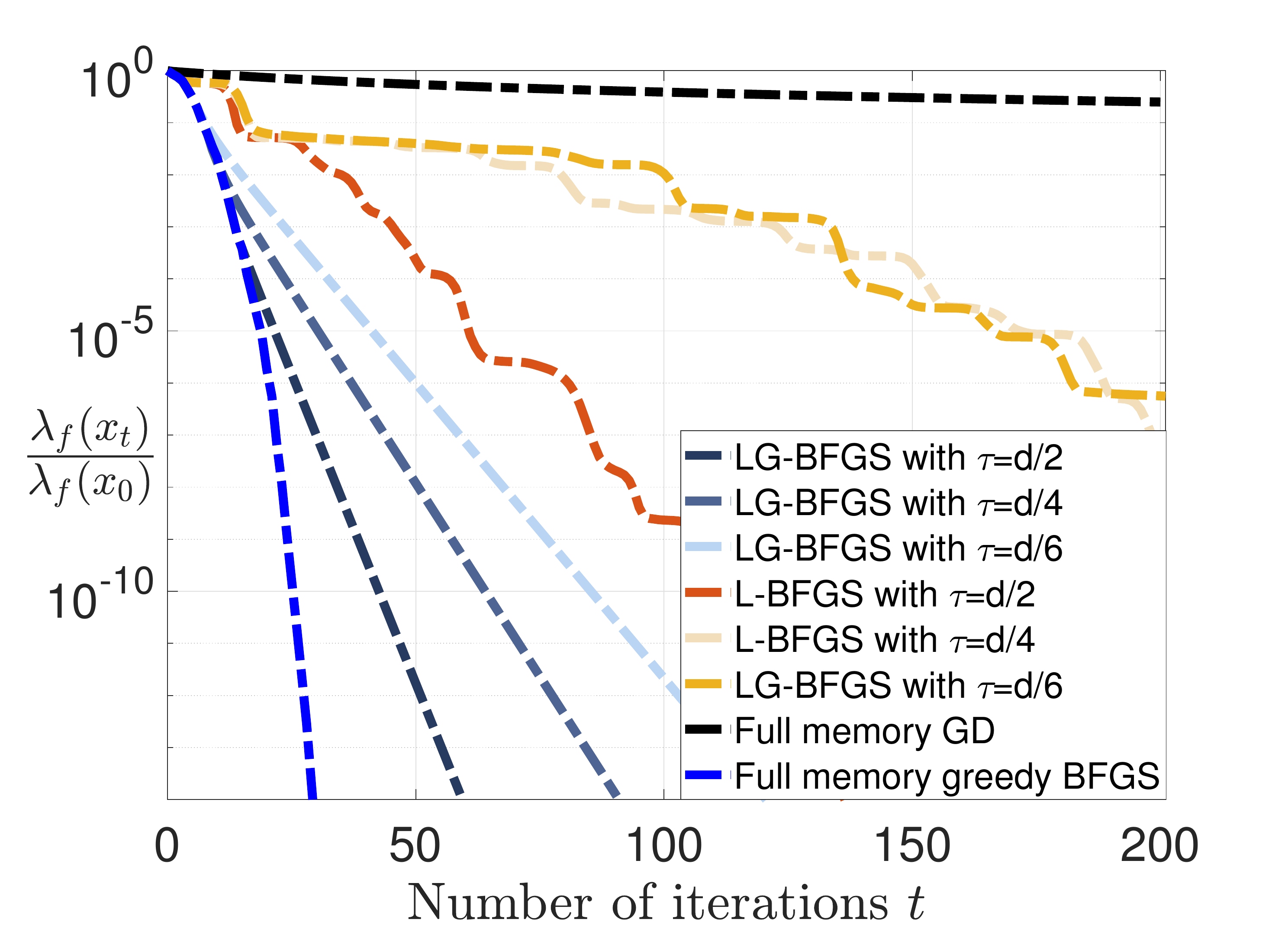}
    \caption{Svmguide3 dataset}
  \end{subfigure}
  \begin{subfigure}{0.42\linewidth}
    \includegraphics[width=\linewidth]{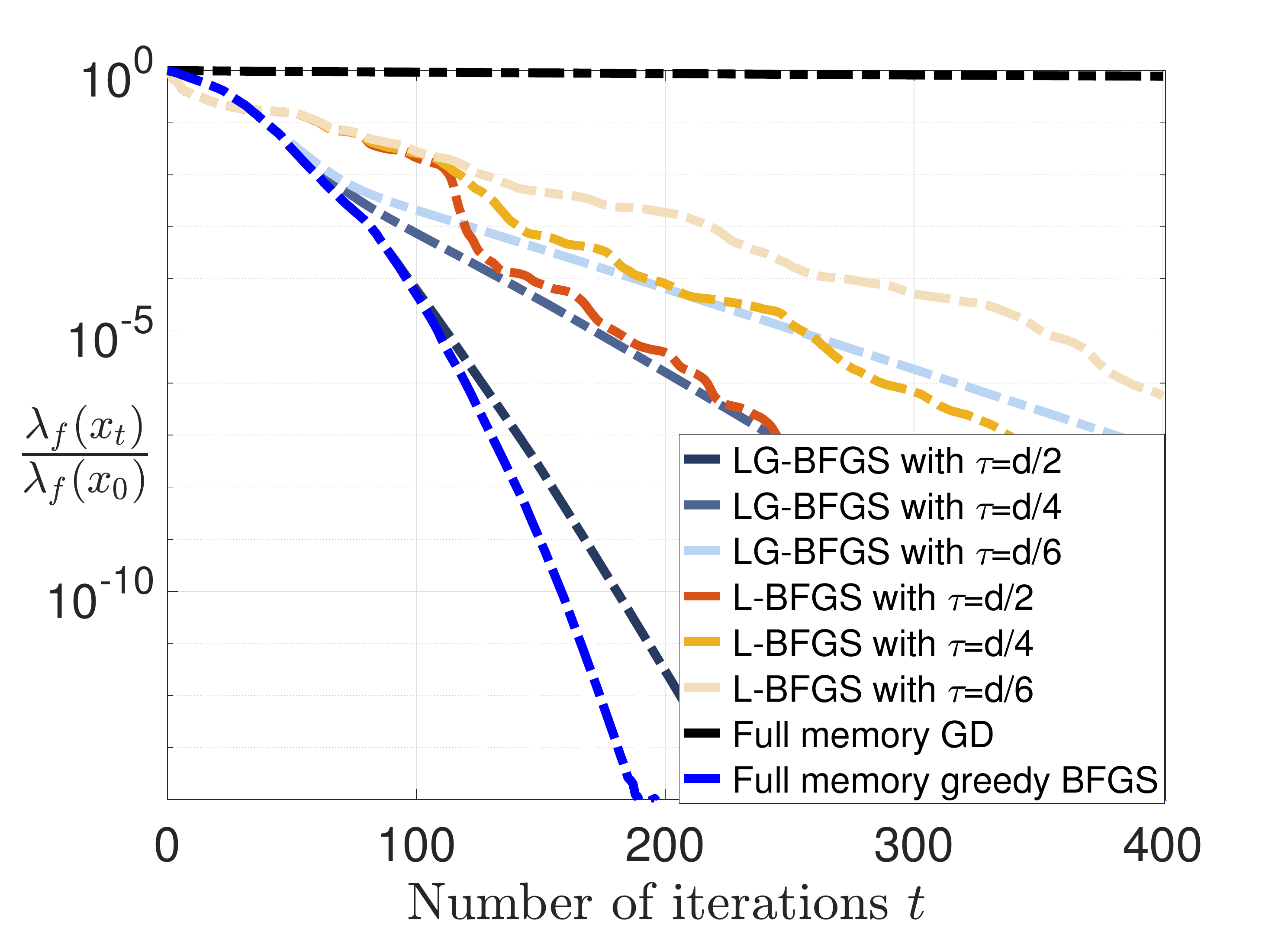}
    \caption{Connect-4 dataset}
  \end{subfigure}
  \begin{subfigure}{0.42\linewidth}
    \includegraphics[width=\linewidth]{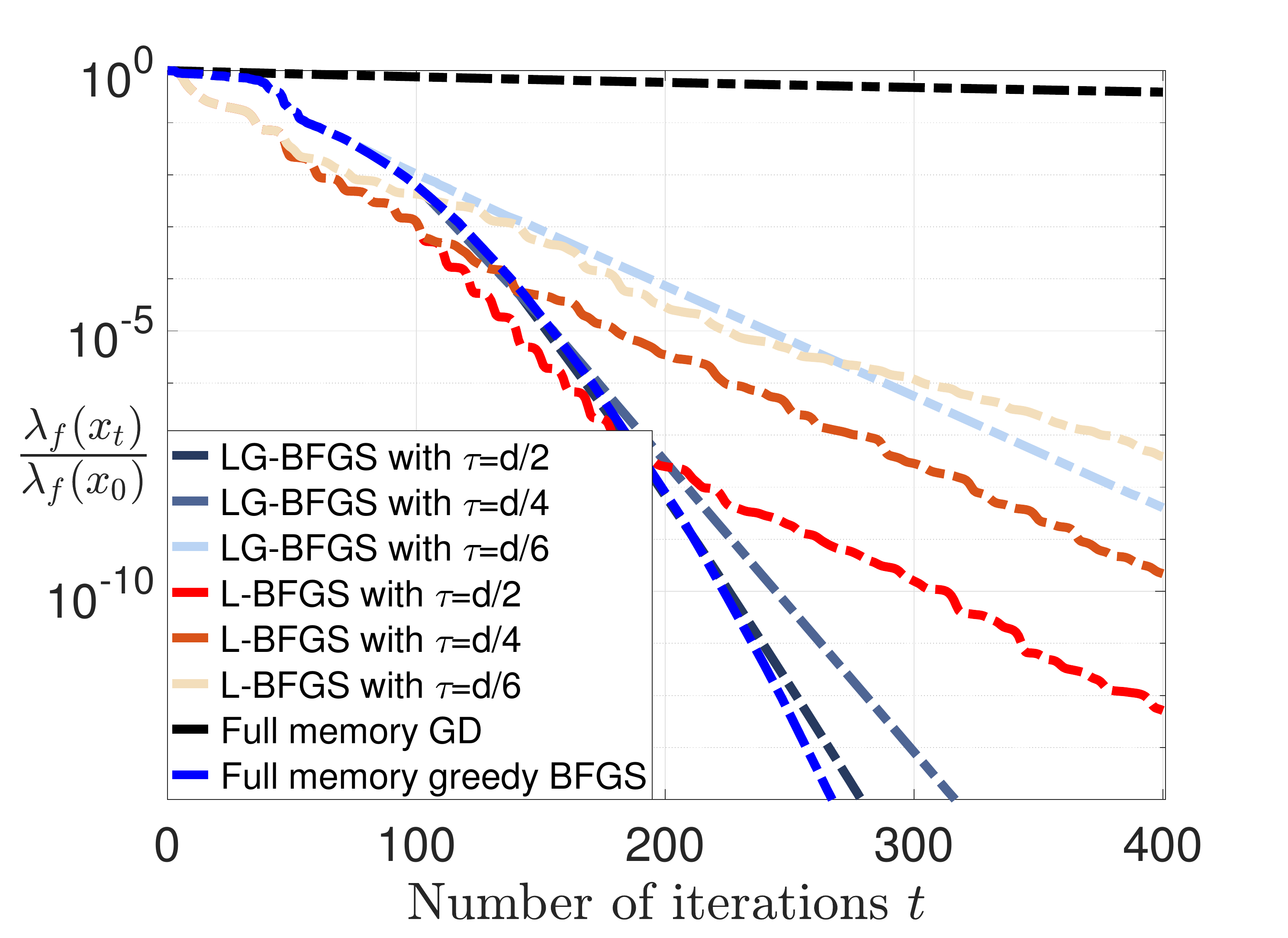}
    \caption{Protein dataset}
  \end{subfigure}
  \begin{subfigure}{0.42\linewidth}
    \includegraphics[width=\linewidth]{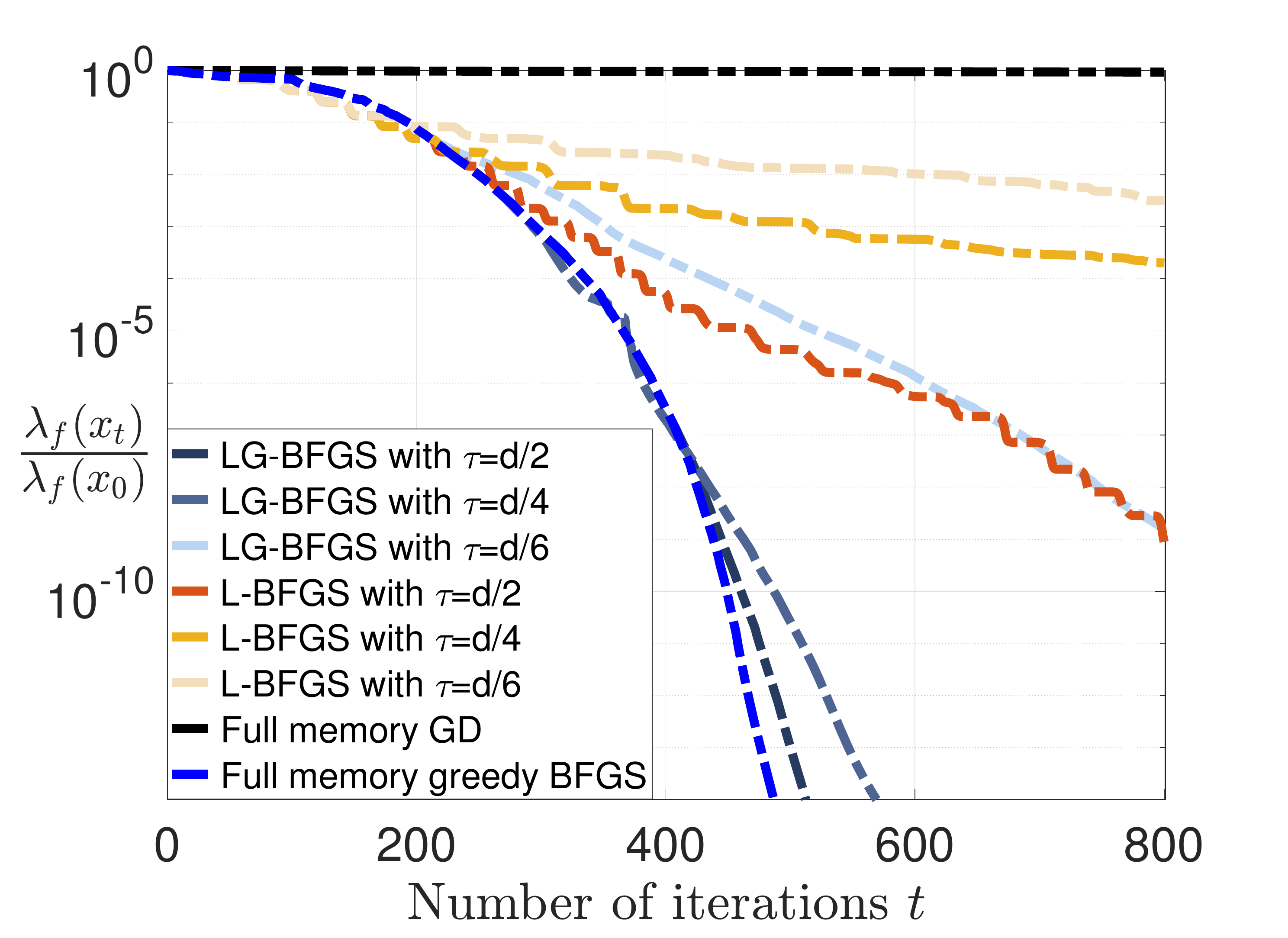}
    \caption{Mnist dataset}
  \end{subfigure}
  \caption{Comparison of LG-BFGS, L-BFGS, gradient descent, and greedy BFGS on four datasets. We consider different memory sizes for LG-BFGS and L-BFGS.}
 \label{fig:1}
\end{figure}

We compare the performance of LG-BFGS with gradient descent, L-BFGS, and greedy BFGS on different datasets. We focus on the logistic regression with $l_2$ regularization where $ f(\bbx) \!=\! \frac{1}{N} \sum_{i=1}^N\ln (1\!+\!\exp(-y_i \bbz_i^\top \bbx)) \!+\! \frac{\mu}{2}\|\bbx\|^2$. Here, $\{\bbz_i\}_{i=1}^N\!\in\! \mathbb{R}^d$ are data points and $\{y_i\}_{i=1}^N\!\in\! \{-1, 1\}$ are labels. We normalize all samples s.t. the objective function gradient is smooth with $L=1/4 + \mu$. Considering the local nature of superlinear results for QN methods, we construct a setup with a warm start, i.e., the initialization is close to the solution. Details 
may be found in Appendix G, where we also analyze performance with a cold start. We run experiments on four datasets: svmguide3, connect-4, protein and mnist, and select the regularization parameter $\mu$ to achieve the best performance, whose descriptions are also 
in Appendix G. We set the stepsize of all QN methods as $1$, while the stepsize of gradient descent is $1/L$. The subset $\{\bbe_i\}_{i=1}^\tau$ is selected: (i) when the number of stored curvature pairs is smaller than $\tau$, we select $\bbs_t$ from the entire basis $\{\bbe_i\}_{i=1}^d$; (ii) when the number of stored curvature pairs reaches $\tau$, we select $\bbs_t$ from the stored variable variations $\{\bbs_u\}_{u=0}^{\tau-1}$.

In Figure \ref{fig:1}, we observe that LG-BFGS consistently outperforms L-BFGS and gradient descent, 
corroborating our theoretical results. 
The convergence rate of LG-BFGS degrades with the decrease of memory size as expected. This is because a smaller memory 
restricts the selection space of the subset $\{\bbe_i\}_{i=1}^\tau$ in \eqref{eq:greedyPreparation}, which increases the minimal relative condition number $\beta_\tau$ in \eqref{eq:propHessianUpdate} and decreases the convergence of the Hessian approximation matrix in \eqref{eq:propHessianUpdate}. Hence, it results in a slower superlinear rate of LG-BFGS and the latter may be close to a linear rate when the memory size is small. Moreover, LG-BFGS can exhibit a comparable performance to greedy BFGS when the memory size is large. This corresponds to the fact that Theorem \ref{thm:superlinearConvergence} recovers the superlinear rate of greedy BFGS when the memory size is the space dimension, i.e., $\tau=d$.

\section{Conclusion}\label{C}
In this work, we developed the LG-BFGS method 
that can exhibit an explicit superlinear convergence rate with a fixed memory size $\tau$, which is smaller than the problem dimension $d$. 
The key attribute of the construction is an innovative synthesis of greedy basis vector selection and displacement aggregation, which conducts greedy selection along a sub-basis and allows ongoing displacement aggregation to ensure memory limitation. The resultant method attains the first non-asymptotic superlinear guarantee for 
limited-memory quasi-Newton methods, which 
explores an explicit trade-off between the memory requirement and the contraction factor in the rate of the convergence. 

\smallskip
\smallskip
\smallskip
{\bf \noindent Disclaimer:} This paper was prepared for informational purposes in part by the Artificial Intelligence Research group of JPMorgan Chase \& Co. and its affiliates (“JP Morgan”), and is not a product of the Research Department of JP Morgan. JP Morgan makes no representation and warranty whatsoever and disclaims all liability, for the completeness, accuracy or reliability of the information contained herein. This document is not intended as investment research or investment advice, or a recommendation, offer or solicitation for the purchase or sale of any security, financial instrument, financial product or service, or to be used in any way for evaluating the merits of participating in any transaction, and shall not constitute a solicitation under any jurisdiction or to any person, if such solicitation under such jurisdiction or to such person would be unlawful.

\section*{Acknowledgements}

The research of A. Mokhtari is supported in part by NSF Grants 2007668, 2019844, and  2112471,  ARO  Grant  W911NF2110226,  the  Machine  Learning  Lab  (MLL)  at  UT  Austin, the Wireless Networking and Communications Group (WNCG) Industrial Affiliates Program, and the NSF AI Institute for Foundations of Machine Learning (IFML).
\newpage

\printbibliography

\newpage

\appendix

\section{Proof of Proposition 1}

For the initial iteration $t=0$ with the initial Hessian approximation $\bbB_0$, curvature pair $\{\bbs_0,\bbr_0\}$ and scaling factor $\psi_1$, consider the corrected Hessian approximation
\begin{align}\label{proof:propeq1}
	\hat{\bbB}_1 = \psi_{1} \bbB_1 = \psi_1 \text{\rm BFGS}(\bbB_{0}, \bbs_{0}, \bbr_{0}).
\end{align}
Define the corrected initial Hessian approximation as $\tilde{\bbB}_0 = \psi_1 \bbB_0$ and the scaled gradient variation as $\tilde{\bbr}_{0} = \psi_1 \bbr_0$. By performing the BFGS update on $\tilde{\bbB}_0$ with $\{\bbs_0, \tilde{\bbr}_{0}\}$, we have 
\begin{align}\label{proof:propeq2}
	\tilde{\bbB}_1 \!=\! \text{\rm BFGS}(\tilde{\bbB}_{0}, \bbs_{0}, \tilde{\bbr}_{0}) & = \tilde{\bbB}_0 \!+\! \frac{\tilde{\bbr}_0 \tilde{\bbr}_0^\top}{\tilde{\bbr}_0^\top \bbs_0} \!-\! \frac{\tilde{\bbB}_0 \bbs_0 \bbs_0^\top \tilde{\bbB}_0^\top}{\bbs_0^\top \tilde{\bbB}_0 \bbs_t} = \psi_1 \bbB_0 + \frac{\psi_1^2 \bbr_0 \bbr_0^\top}{\psi_1 \bbr_0^\top \bbs_0} - \frac{\psi_1^2 \bbB_0 \bbs_0 \bbs_0^\top \bbB_0^\top}{ \psi_1 \bbs_0^\top \bbB_0 \bbs_t} \\
	& =\! \psi_1 \bbB_0 \!+\! \frac{\psi_1 \bbr_0 \bbr_0^\top}{ \bbr_0^\top \bbs_0} \!-\! \frac{\psi_1 \bbB_0 \bbs_0 \bbs_0^\top \bbB_0^\top}{ \bbs_0^\top \bbB_0 \bbs_t} = \psi_1 \Big( \bbB_0 + \frac{\bbr_0 \bbr_0^\top}{ \bbr_0^\top \bbs_0} - \frac{\bbB_0 \bbs_0 \bbs_0^\top \bbB_0^\top}{ \bbs_0^\top \bbB_0 \bbs_t} \Big) \nonumber \\  
	& = \psi_1 \text{\rm BFGS}(\bbB_{0}, \bbs_{0}, \bbr_{0}) = \psi_{1} \bbB_1 = \hat{\bbB}_1. \nonumber
\end{align}
This indicates that scaling the Hessian approximation matrix $\bbB_1$ is equivalent to scaling the initial Hessian approximation $\bbB_0$ and the gradient variation $\bbr_0$. Thus, the proposition conclusion holds for the initial iteration $t=1$.

For any iteration $t > 1$ with the initial Hessian approximation $\bbB_0$, stored curvature pair $\{\bbs_u,\bbr_u\}_{u=0}^{t-1}$ and the scaling factor $\psi_{t}$, 
consider the corrected Hessian approximation
\begin{align}\label{proof:propeq3}
	\hat{\bbB}_t = \psi_{t} \bbB_{t} = \psi_{t-1} \text{\rm BFGS}(\bbB_{t-1}, \bbs_{t-1}, \bbr_{t-1}),~\ldots, ~\bbB_1 = \text{\rm BFGS}(\bbB_{0}, \bbs_{0}, \bbr_{0}).
\end{align}
Define the corrected initial Hessian approximation as $\tilde{\bbB}_0 = \psi_{t} \bbB_0$ and the scaled gradient variations as $\tilde{\bbr}_{u} = \psi_{t} \bbr_u$ for $u=0,\ldots,t-1$. By performing the BFGS updates on $\hat{\bbB}_0$ with $\{\bbs_u, \tilde{\bbr}_u\}_{u=0}^{t-1}$, we have
\begin{align}\label{proof:propeq4}
	\tilde{\bbB}_t = \text{\rm BFGS}(\tilde{\bbB}_{t-1}, \bbs_{t-1}, \tilde{\bbr}_{t-1}),~\ldots, ~\tilde{\bbB}_1 = \text{\rm BFGS}(\tilde{\bbB}_{0}, \bbs_{0}, \tilde{\bbr}_{0}).
\end{align}
We now use \textbf{induction} to prove the following statement 
\begin{align}\label{proof:propeq45}
    \tilde{\bbB}_k = \psi_{t}\bbB_k,~\forall~k=1,\ldots,t.
\end{align}
For the initial iteration $k=1$, by performing the BFGS update on $\tilde{\bbB}_0$ with $\{\bbs_0, \tilde{\bbr}_{0}\}$, we have
\begin{align}\label{proof:propeq5}
	&\tilde{\bbB}_1 \!\!=\! \text{\rm BFGS}(\tilde{\bbB}_{0},\! \bbs_{0},\! \tilde{\bbr}_{0}) \!=\! \psi_{t} \Big(\! \bbB_0 \!+\! \frac{\bbr_0 \bbr_0^\top}{ \bbr_0^\top \bbs_0} \!-\! \frac{\bbB_0 \bbs_0 \bbs_0^\top \bbB_0^\top}{ \bbs_0^\top \bbB_0 \bbs_0} \!\Big) \!=\! \psi_{t} \text{\rm BFGS}(\bbB_{0},\! \bbs_{0},\! \bbr_{0}) \!=\! \psi_{t} \bbB_1.
\end{align}
Thus, \eqref{proof:propeq45} holds for $k=1$. Assume \eqref{proof:propeq45} holds for iteration $k-1 \ge 1$, i.e., $\tilde{\bbB}_{k-1} = \psi_{t} \bbB_{k-1} $, and consider iteration $k$. By performing the BFGS update on $\tilde{\bbB}_{k-1}$ with $\{\bbs_{k-1}, \tilde{\bbr}_{k-1}\}$, we have
\begin{align}\label{proof:propeq6}
	\tilde{\bbB}_k \!=\! \text{\rm BFGS}(\tilde{\bbB}_{k\!-\!1}, \bbs_{k\!-\!1}, \tilde{\bbr}_{k\!-\!1}) \!=\! \psi_{t} \Big( \bbB_{k\!-\!1} \!+\! \frac{\bbr_{k\!-\!1} \bbr_{k-1}^\top}{ \bbr_{k-1}^\top \bbs_{k-1}} \!-\! \frac{\bbB_{k-1} \bbs_{k-1} \bbs_{k-1}^\top \bbB_{k-1}^\top}{ \bbs_{k-1}^\top \bbB_{k-1} \bbs_{k-1}} \!\Big) 
    \!=\! \psi_{t} \bbB_k 
\end{align}
where $\tilde{\bbB}_{k-1} = \psi_{t} \bbB_{k-1}$ is used in the second equality. By combining \eqref{proof:propeq5} and \eqref{proof:propeq6}, we prove \eqref{proof:propeq45} by induction. Thus, we get
\begin{align}
	\tilde{\bbB}_t = \psi_{t} \bbB_{t} = \hat{\bbB}_t.
\end{align}
This indicates that scaling the Hessian approximation matrix $\bbB_{t}$ is equivalent to scaling the initial Hessian approximation $\bbB_0$ and the gradient variations $\{\bbr_u\}_{u=0}^{t-1}$. 

We conclude that at each iteration $t$, scaling the Hessian approximation matrix $\bbB_t$ by $\psi_{t}$ is equivalent to scaling the initial Hessian approximation $\bbB_0$ and the gradient variations $\{\bbr_u\}_{u=0}^{t-1}$ by $\psi_{t}$. Therefore, we can incorporate the correction strategy into the displacement step by scaling 
the gradient variations and maintain the remaining unchanged, which completes the proof.


\section{Proof of Proposition 2}\label{appendixB}

\noindent We need the following lemmas to complete the proof.
\begin{lemma}\label{lemma:1}
	If LG-BFGS and greedy BFGS perform the greedy selection from the same subset $\{\bbe_i\}_{i=1}^\tau$ of size $\tau$ and have the same initial settings, the iterates $\{\bbx_{L,t}\}_{t}$ generated by LG-BFGS equal to the iterates $\{\bbx_{G, t}\}_t$ generated by greedy BFGS. 
\end{lemma}
\begin{proof}
	We start by noting that greedy BFGS updates the variable with (1) and the Hessian inverse approximation with (3). This is equivalent to updating the variable and the Hessian inverse approximation from the initial Hessian inverse approximation $\bbH_0$ with all historical curvature pairs $\{\bbs_k, \bbr_k\}_{k=0}^{t-1}$ at each iteration $t$. In this context, we can prove the lemma by proving the iterate $\bbx_{L,t}$ generated by LG-BFGS equal to the iterate $\bbx_{G, t}$ generated from the initial Hessian inverse approximation $\bbH_0$ with all historical curvature pairs $\{\bbs_k, \bbr_k\}_{k=0}^{t-1}$ for any iteration $t \ge 0$.\footnote{Without loss of generality, we assume $\{\}_{a}^b = \emptyset$, $\sum_{a}^b = 0$ and $\prod_{a}^b = 1$ if $b < a$.}
	
	Specifically, we use \textbf{induction} to prove the lemma. 
	At the initial iteration $t=0$, this conclusion holds because LG-BFGS and greedy BFGS have the same initial setting $\bbx_{L,0} = \bbx_{G,0}$. Assume that the conclusion holds at iteration $t-1 \ge 0$, i.e., the iterate $\bbx_{L,t-1}$ generated by LG-BFGS with the limited-memory curvature pairs $\ccalP_{t-1}$ equal to the iterate $\bbx_{G,t-1}$ generated by greedy BFGS with all historical curvature pairs $\{\bbs_k, \bbr_k\}_{k=0}^{t-2}$ as 
	\begin{align}\label{proof:lemma1eq1}
		\bbx_{L,t-1} = \bbx_{G,t-1}.
	\end{align}

	Consider iteration $t$ with the new curvature pair $\{\bbs_{t-1}, \bbr_{t-1}\}$. Greedy BFGS updates the historical curvature pairs by adding the new curvature pair $\{\bbs_{t-1}, \bbr_{t-1}\}$ directly and form the new historical curvature pairs $\{\bbs_k, \bbr_k\}_{k=0}^{t-1}$. LG-BFGS updates the curvature pairs by incorporating the information $\{\bbs_{t-1}, \bbr_{t-1}\}$ into $\ccalP_{t-1}$ and form the new curvature pairs $\ccalP_t$. From Theorem 3.2 in \citep{berahas2022limited}, if the Hessian inverse approximation generated from $\bbH_0$ with $\ccalP_{t-1}$ equal to that generated from $\bbH_0$ with $\{\bbs_k, \bbr_k\}_{k=0}^{t-2}$, the Hessian inverse approximation generated from $\bbH_0$ with $\ccalP_{t}$ equal to that generated from $\bbH_0$ with $\{\bbs_k, \bbr_k\}_{k=0}^{t-1}$. By using this result and \eqref{proof:lemma1eq1}, we get 
	\begin{align}\label{proof:lemma1eq2}
		\bbx_{L,t} = \bbx_{G,t}.
	\end{align}
	By combining \eqref{proof:lemma1eq1} and \eqref{proof:lemma1eq2}, we prove by induction that $\bbx_{L,t} = \bbx_{G,t}$ for any iteration $t \ge 0$, which completes the proof.
\end{proof}
\begin{lemma}[Lemma 4.3 in \citep{rodomanov2021greedy}]\label{lemma:2}
	Let $\bbx$ be a decision variable and $\bbB$ the Hessian approximation satisfying
	\begin{align}
		\nabla^2 f(\bbx) \preceq \bbB \preceq \eta \nabla^2 f(\bbx)
	\end{align}
	for some $\eta \ge 1$. Let also $\bbx_+$ be the updated decision variable as
	\begin{align}\label{eq:update}
		\bbx_+ = \bbx - \bbB^{-1} \nabla f(\bbx)
	\end{align}
	and $\lambda_f(\bbx)$ be such that $\lambda_f(\bbx) C_M \le 2$. Then, it holds that
	\begin{align}\label{eq:lemma21}
		\phi \!=\! \|\bbx_+ \!-\! \bbx\|_{\nabla^2 f(\bbx)} \!\le\! \lambda_f(\bbx)~~\text{\rm and}~~\lambda_f(\bbx_+) \!\le\! \Big(1\!+\! \frac{\lambda_f(\bbx) C_M}{2}\Big) \frac{\eta -1 + \frac{\lambda_f(\bbx) C_M}{2} }{\eta} \lambda_f(\bbx).
	\end{align}
\end{lemma}
\begin{lemma}[Lemma 4.4 in \citep{rodomanov2021greedy}]\label{lemma:3}
	Let $\bbx$ be a decision variable and $\bbB$ the Hessian approximation satisfying
	\begin{align}
		\nabla^2 f(\bbx) \preceq \bbB \preceq \eta \nabla^2 f(\bbx)
	\end{align}
	for some $\eta \ge 1$. Let also $\bbx_+$ be the updated decision variable [cf. \eqref{eq:update}] and $\phi = \|\bbx_+ - \bbx \|_{\nabla^2 f(\bbx)}$ be the weighted update difference. Then, it holds that
	\begin{align}\label{eq:lemma31}
		\nabla^2 f(\bbx_+) \preceq (1+C_M \phi) \bbB = \hat{\bbB}
	\end{align}
	and the Hessian approximation $\bbB_+$ updated by the BFGS on $\hat{\bbB}$ with the curvature pair $\{\bbs,\bbr\}$ satisfies
	\begin{align}\label{eq:lemma32}
		\nabla^2 f(\bbx_+) \preceq \text{\rm BFGS}\big(\hat{\bbB}, \bbs, \bbr \big) \preceq \eta (1 + C_M \phi)^2 \nabla^2 f(\bbx_+).
	\end{align}
\end{lemma}

\begin{proof}[Proof of Proposition 2]
	From Lemma \ref{lemma:1}, we know that the iterates generated by LG-BFGS is equivalent to the iterates generated by greedy BFGS, if both perform greedy selection in the same subset $\{\bbe_i\}_{i=1}^\tau$ of memory size $\tau$. In this context, we can prove the linear convergence of the iterates generated by LG-BFGS by proving the linear convergence of the iterates generated by the corresponding greedy BFGS, alternatively. 	
	
	We start by defining the concise notation $\lambda_t = \lambda_f(\bbx_t)$, $\phi_t = \|\bbx_{t+1} - \bbx_t\|_{\nabla^2 f(\bbx_t)}$ and 
	\begin{align}\label{eq:etaRepresentation}
		\eta_t = e^{2C_M \sum_{k=0}^{t-1}\lambda_k}\frac{L}{\mu}
	\end{align}
	for convenience of expression. We use \textbf{induction} to prove the following statement
	\begin{align}\label{proof:prop2eq0}
		\nabla^2 f(\bbx_t) &\preceq \bbB_t \preceq \eta_t \nabla^2 f(\bbx_t), \\
		\label{proof:prop2eq05}\lambda_t &\le (1 - \frac{\mu}{2 L})^t \lambda_0
	\end{align}
	for any iteration $t \ge 0$. For the initial iteration $t=0$ with the initial condition, we have 
	\begin{align}\label{proof:prop2eq1}
		\nabla^2 f(\bbx_0) \preceq \bbB_0 \preceq \frac{L}{\mu}\nabla^2 f(\bbx_0) = \eta_0 \nabla^2 f(\bbx_0)
	\end{align}
	and 
	\begin{align}\label{proof:prop2eq2}
		\lambda_0 \le (1 - \frac{\mu}{2 L})^0 \lambda_0 = \lambda_0.
	\end{align}
	Thus, \eqref{proof:prop2eq0} and \eqref{proof:prop2eq05} hold for $t=0$.
	
	Assume that for iteration $t-1 \ge 0$, we have 
	\begin{align}\label{proof:prop2eq3}
		\nabla^2 f(\bbx_k) &\preceq \bbB_k \preceq \eta_k \nabla^2 f(\bbx_k),\\
		\label{proof:prop2eq35} \lambda_k &\le (1 - \frac{\mu}{2 L})^{k} \lambda_0
	\end{align}
	for all $0 \le k \le t-1$, and consider iteration $t$. By using Lemma \ref{lemma:2} with the condition \eqref{proof:prop2eq3}, we have
	\begin{align}\label{proof:prop2eq4}
		\lambda_t \le \Big(1+ \frac{\lambda_{t-1} C_M}{2}\Big) \frac{\eta_{t-1} -\left(1 - \frac{\lambda_{t-1} C_M}{2}\right) }{\eta_{t-1}} \lambda_{t-1}.
	\end{align}
	By using the fact $C_M \lambda_{t-1} \le C_M \lambda_0 \le 1$ from the initial condition and the inequality $1-x \ge e^{-2t}$ for any $0 \le x \le 1/2$, we have
	\begin{align}\label{proof:prop2eq5}
		\frac{1 - \frac{\lambda_{t-1} C_M}{2}}{\eta_{t-1}} \ge \frac{e^{-\lambda_{t-1} C_M}}{\eta_{t-1}}. 
	\end{align}
	By substituting the representation of $\eta_{t-1}$ into \eqref{proof:prop2eq5}, we get
	\begin{align}\label{proof:prop2eq6}
		\frac{1 - \frac{\lambda_{t-1} C_M}{2}}{\eta_{t-1}} \ge e^{-C_M \lambda_{t-1} - 2 C_M \sum_{k=0}^{t-2} \lambda_k} \frac{\mu}{L} \ge e^{- 2 C_M \sum_{k=0}^{t-1} \lambda_k} \frac{\mu}{L}.
	\end{align}
	The term $2 C_M \sum_{k=0}^{t-1}\lambda_k$ in \eqref{proof:prop2eq6} can be bounded as
	\begin{align}\label{proof:prop2eq7}
		2 C_M \sum_{k=0}^{t-1} \lambda_i \le 2 C_M \sum_{k=0}^{t-1} (1-\frac{\mu}{2L})^k \lambda_0 \le \frac{4L}{\mu}C_M \lambda_0 \le \ln \frac{3}{2}
	\end{align}
	where the condition \eqref{proof:prop2eq35} is used in the second inequality and the initial condition is used in the last inequality. By substituting \eqref{proof:prop2eq7} into \eqref{proof:prop2eq6}, we have
	\begin{align}\label{proof:prop2eq8}
		\frac{1 - \frac{\lambda_{t-1} C_M}{2}}{\eta_{t-1}} \ge \frac{2\mu}{3L}.
	\end{align}
	From the condition \eqref{proof:prop2eq35} and the initial condition, we get
	\begin{align}\label{proof:prop2eq9}
		\frac{\lambda_{t-1} C_M}{2} \le \frac{\lambda_{0} C_M}{2} \le \frac{\ln \frac{3}{2} \ \mu }{8 L} \le \frac{\mu}{16 L}
	\end{align}
	where the inequality $\ln (1 + x) \le x$ for any $x \ge 0$ is used in the last inequality. By substituting \eqref{proof:prop2eq9} and \eqref{proof:prop2eq8} into \eqref{proof:prop2eq4}, we have
	\begin{align}\label{proof:prop2eq10}
		\lambda_t \le \Big(1+ \frac{\mu}{16 L}\Big) \Big( 1 - \frac{2\mu}{3L} \Big) \lambda_{t-1} \le \Big( 1 - \frac{\mu}{2L} \Big) \lambda_{t-1} \le \Big( 1 - \frac{\mu}{2L} \Big)^t \lambda_0
	\end{align}
 where the condition \eqref{proof:prop2eq35} is used in the last inequality. By using Lemma \ref{lemma:3} with the condition \eqref{proof:prop2eq3}, we have
	\begin{align}\label{proof:prop2eq11}
		&\nabla^2 f(\bbx_{t}) \preceq \bbB_{t} \preceq (1+\phi_{t-1} C_M)^2 \eta_{t-1} \nabla^2 f(\bbx_{t}).
	\end{align}
	By using the result $\phi_{t-1} \le \lambda_{t-1}$ from Lemma \ref{lemma:2} and the inequality $(1+x) \le e^{2x}$, we get
	\begin{align}\label{proof:prop2eq12}
		\bbB_{t} &\preceq (1+\lambda_{t-1} C_M)^2 \eta_{t-1} \nabla^2 f(\bbx_{t}) \preceq e^{2C_M \lambda_{t-1}} \eta_{t-1} \nabla^2 f(\bbx_{t+1}).
	\end{align}
	By further substituting the representation of $\eta_{t-1}$ into \eqref{proof:prop2eq12}, we have
	\begin{align}\label{proof:prop2eq13}
		\nabla^2 f(\bbx_{t}) \preceq \bbB_{t} \preceq e^{2 C_M \sum_{k=0}^{t-1}\lambda_k}\frac{L}{\mu}\nabla^2 f(\bbx_{t}) = \eta_{t} \nabla^2 f(\bbx_{t}). 
	\end{align}
	By combining \eqref{proof:prop2eq1}, \eqref{proof:prop2eq2}, \eqref{proof:prop2eq10} and \eqref{proof:prop2eq13}, we prove \eqref{proof:prop2eq0} and \eqref{proof:prop2eq05} by induction, which completes the proof.
\end{proof}

\section{Proof of Proposition 3}

We need the following lemma to complete the proof.

\begin{lemma}[Lemma 2.4 in \citep{rodomanov2021greedy}]\label{lemma:4}
	Consider two positive definite matrices $\bbA \preceq \bbD$. For any vector $\bbs \in \mathbb{R}^d$, it holds that
	\begin{align}
		\sigma(\bbA, \bbD) - \sigma(\bbA, {\rm BFGS}(\bbD, \bbs, \bbA \bbs)) \ge \frac{\bbs^\top (\bbD - \bbA) \bbs}{\bbs^\top \bbA \bbs}
	\end{align}
     where ${\rm BFGS}(\bbB, \bbs, \bbA \bbs)$ is the BFGS update on $\bbB$ with the curvature pair $\{\bbs, \bbA\bbs\}$.
\end{lemma}

\begin{proof}[Proof of Proposition 3]
	From Lemma \ref{lemma:3}, we know that $\nabla^2 f(\bbx_+) \preceq \hat{\bbB}$. Let $\bbB_+ = \text{\rm BFGS}(\hat{\bbB}, \bbs, \bbr)$ be the updated Hessian approximation matrix, where $\{\bbs, \bbr\}$ are the curvature pair selected greedily from the subset $\{\bbe_i\}_{i=1}^\tau$, i.e.,
	\begin{align}\label{proof:prop3eq1}
		\bbs = \argmax_{\bbs \in \{\bbe_1,\ldots,\bbe_\tau\}} \frac{\bbs^\top \hat{\bbB} \bbs}{\bbs^\top \nabla^2 f(\bbx_+) \bbs}.
	\end{align}
	Denote by $\sigma_{\bbx_+}(\bbB_+)$, $\sigma_{\bbx_+}(\hat{\bbB})$ and $\sigma(\bbB)$ the concise notation of $\sigma(\nabla^2 f(\bbx_+), \bbB_+)$, $\sigma(\nabla^2 f(\bbx_+), \hat{\bbB})$ and $\sigma(\nabla^2 f(\bbx), \bbB)$. By using Lemma \ref{lemma:4} with $\bbA = \nabla^2 f(\bbx_+)$ and $\bbD = \hat{\bbB}$, we have
	\begin{align}\label{proof:prop3eq2}
		\sigma_{\bbx_+}(\hat{\bbB}) - \sigma_{\bbx_+}(\bbB_+) \ge \frac{\bbs^\top\big(\hat{\bbB} - \nabla^2 f(\bbx_+) \big) \bbs}{\bbs^\top \nabla^2 f(\bbx_+)\bbs}.
	\end{align}
	By substituting \eqref{proof:prop3eq1} into \eqref{proof:prop3eq2}, we have
	\begin{align}\label{proof:prop3eq3}
		\sigma_{\bbx_+}(\hat{\bbB}) - \sigma_{\bbx_+}(\bbB_+) \ge \max_{1 \le i \le \tau} \frac{\bbe_i^\top \big(\hat{\bbB} - \nabla^2 f(\bbx_+) \big) \bbe_i}{\bbe_i^\top \nabla^2 f(\bbx_+)\bbe_i}. 
	\end{align}
	Let $\bbE = \hat{\bbB} - \nabla^2 f(\bbx_+)$ be the approximation error matrix. From Assumption 1, we have
	\begin{align}\label{proof:prop3eq4}
		\mu \bbI \preceq \nabla^2 f(\bbx_+) \preceq L \bbI
	\end{align}
	where $\bbI$ is the identity matrix. Substituting \eqref{proof:prop3eq4} into \eqref{proof:prop3eq3} yields
	\begin{align}\label{proof:prop3eq5}
		\sigma_{\bbx_+}(\hat{\bbB}) - \sigma_{\bbx_+}(\bbB_+) \ge \frac{1}{L}\max_{1 \le i \le \tau} \bbe_i^\top \bbE \bbe_i. 
	\end{align}
	Let $\beta(\bbe_i)$ be the relative condition number of the basis vector $\bbe_i$ w.r.t. $\bbE$ for $i=1,\ldots,\tau$. From the definition of the relative condition number [Def. 1], we have
	\begin{align}\label{proof:prop3eq6}
		\bbe_i^\top \bbE \bbe_i = \frac{1}{\beta(\bbe_i)}\max_{1 \le i \le d} \bbe_i^\top \bbE \bbe_i.
	\end{align}
	By substituting \eqref{proof:prop3eq6} into \eqref{proof:prop3eq5}, we have
	\begin{align}\label{proof:prop3eq7}
		\sigma_{\bbx_+}(\hat{\bbB}) \!-\! \sigma_{\bbx_+}(\bbB_+) &\ge\! \frac{1}{L}\!\max_{1 \le i \le \tau}\! \Big(\frac{1}{\beta(\bbe_i)}\max_{1 \le i \le d} \bbe_i^\top \bbE \bbe_i \Big) \\ 
		&= \frac{1}{L \min_{1 \le i \le \tau} \beta(\bbe_i)}\max_{1 \le i \le d} \bbe_i^\top \bbE \bbe_i = \frac{1}{L \beta_\tau}\max_{1 \le i \le d} \bbe_i^\top \bbE \bbe_i \nonumber
	\end{align}
	where $\beta_\tau$ is the minimal relative condition number of the subset $\{\bbe_i\}_{i=1}^\tau$ w.r.t. $\bbE$. Since $\max_{1 \le i \le d} \bbe_i^\top \bbE \bbe_i \ge \bbe_i^\top \bbE \bbe_i$ for all $i=1,\ldots,d$, we get
	\begin{align}\label{proof:prop3eq8}
		\sigma_{\bbx_+}(\hat{\bbB}) \!-\! \sigma_{\bbx_+}(\bbB_+) \ge \frac{1}{L \beta_\tau d}\sum_{i=1}^d \bbe_i^\top \bbE \bbe_i = \frac{1}{L \beta_\tau d}\sum_{i=1}^d {\rm Tr} (\bbe_i \bbe_i^\top, \bbE)
	\end{align}
	where $\rm Tr(\cdot, \cdot)$ represents the trace operation. From the linearity of the trace operation, we get
	\begin{align}\label{proof:prop3eq9}
		\sigma_{\bbx_+}(\hat{\bbB}) \!-\! \sigma_{\bbx_+}(\bbB_+) \!\ge\! \frac{1}{\beta_\tau L d} {\rm Tr} \big(\sum_{i=1}^d\! \bbe_i \bbe_i^\top\!, \bbE\big) \!=\! \frac{1}{\beta_\tau L d} {\rm Tr} (\bbI, \bbE) \!\ge\! \frac{\mu}{\beta_\tau L d} {\rm Tr} (\nabla^2 f(\bbx_+)^{-1}\!\!, \bbE)
	\end{align}
	where the condition \eqref{proof:prop3eq4} is used in the last inequality. From the definition $\sigma_{\bbx_+}(\hat{\bbB}) = {\rm Tr} (\nabla^2 f(\bbx_+)^{-1}\!, \bbE)$, we have
	\begin{align}\label{proof:prop3eq10}
		\sigma_{\bbx_+}(\bbB_+) \le (1 - \frac{\mu}{\beta_\tau L d}) \sigma_{\bbx_+}(\hat{\bbB}).
	\end{align}
	We then characterize the relationship between $\sigma_{\bbx_+}(\hat{\bbB})$ and $\sigma_{\bbx}(\bbB)$. We can represent $\sigma_{\bbx_+}(\hat{\bbB})$ by definition as
	\begin{align}\label{proof:prop3eq11}
		\sigma_{\bbx_+}\!(\hat{\bbB}) = {\rm Tr} \big(\nabla^2 f(\bbx_+)^{-1}\hat{\bbB}\big) - d = (1 + \phi C_M) {\rm Tr} \big(\nabla^2 f(\bbx_+)^{-1} \bbB\big) - d
	\end{align}
	where $\hat{\bbB} = (1 + \phi C_M) \bbB$ is used in the last equality. Since $1 + \phi C_M \ge 1$, we can upper bound \eqref{proof:prop3eq11} as
	\begin{align}\label{proof:prop3eq12}
		\sigma_{\bbx_+}\!(\hat{\bbB}) \le (1 + \phi C_M)^2{\rm Tr} \big(\nabla^2 f(\bbx_+)^{-1} \bbB\big) - d = (1 + \phi C_M)^2(\sigma_{\bbx}(\bbB) + d) - d.
	\end{align}
	Expanding the terms in \eqref{proof:prop3eq12} yields
	\begin{align}\label{proof:prop3eq13}
		\sigma_{\bbx_+}\!(\hat{\bbB}) &\le (1 + \phi C_M)^2\sigma_{\bbx}(\bbB) + d\big((1 + \phi C_M)^2 -1\big)\\
		&= (1 + \phi C_M)^2\sigma_{\bbx}(\bbB) + 2 d \phi C_M \Big( 1 + \frac{\phi C_M}{2} \Big) \nonumber\\
		&\le (1 + \phi C_M)^2\Big(\sigma_{\bbx}(\bbB) + \frac{2d \phi C_M}{1 + \phi C_M} \Big). \nonumber
	\end{align}
	By substituting \eqref{proof:prop3eq13} into \eqref{proof:prop3eq10}, we complete the proof
	\begin{align}\label{proof:prop3eq14}
		\sigma_{\bbx_+}\!(\bbB_+) \le \Big(1 - \frac{\mu}{\beta_\tau d L}\Big)(1 + \phi C_M)^2\Big(\sigma_{\bbx}(\bbB) + \frac{2d\phi C_M}{1+\phi C_M} \Big).
	\end{align}
\end{proof}

\section{Proof of Theorem 1}

From Lemma \ref{lemma:1}, we know that the iterates generated by LG-BFGS is equivalent to the iterates generated by greedy BFGS, if both perform greedy selection in the same subset $\{\bbe_i\}_{i=1}^\tau$ of memory size $\tau$. In this context, we can prove the superlinear convergence of the iterates generated by LG-BFGS by proving the superlinear convergence of the iterates generated by the corresponding greedy BFGS, alternatively. 	

Denote by $\lambda_t$ and $\sigma_t$ the concise notation of $\lambda_f(\bbx_t)$ and $\sigma(\nabla^2 f(\bbx_t), \bbB_t)$. We start by noting that 
\begin{align}\label{proof:thm1eq1}
	\sigma_t = {\rm Tr}(\nabla^2 f(\bbx_t)^{-1} \bbB_t) - d = {\rm Tr}\big(\nabla^2 f(\bbx_t)^{-1} (\bbB_t - \nabla^2 f(\bbx_t) \big)
\end{align}
where $\bbB_t - \nabla^2 f(\bbx_t)$ is positive semidefinite from Proposition 2 [cf. \eqref{proof:prop2eq0}]. Since the maximal eigenvalue of $\nabla^2 f(\bbx_t)^{-1} \big(\bbB_t - \nabla^2 f(\bbx_t)\big)$ is bounded by the trace of $\nabla^2 f(\bbx_t)^{-1} (\bbB_t - \nabla^2 f(\bbx_t)$, i.e., the sum of eigenvalues of $\nabla^2 f(\bbx_t)^{-1} (\bbB_t - \nabla^2 f(\bbx_t)$, we have
\begin{align}\label{proof:thm1eq2}
	\nabla^2 f(\bbx_t)^{-1} \big(\bbB_t - \nabla^2 f(\bbx_t)\big) \preceq {\rm Tr}\big(\nabla^2 f(\bbx_t)^{-1} (\bbB_t - \nabla^2 f(\bbx_t) \big) \bbI.
\end{align}
By multiplying positive definite matrix $\nabla^2 f(\bbx_t)$ on both sides of \eqref{proof:thm1eq2}, we have
\begin{align}\label{proof:thm1eq3}
	\bbB_t \preceq \Big( 1+ {\rm Tr}\big(\nabla^2 f(\bbx_t)^{-1} (\bbB_t - \nabla^2 f(\bbx_t) \big)\Big) \nabla^2 f(\bbx_t) = ( 1+ \sigma_t) \nabla^2 f(\bbx_t)
\end{align}
and
\begin{align}\label{proof:thm1eq4}
	\nabla^2 f(\bbx_t) \preceq \bbB_t \preceq (1+ \sigma_t)\nabla^2 f(\bbx_t).
\end{align} 
By using Lemma \ref{lemma:2} with the condition \eqref{proof:thm1eq4}, we have
\begin{align}\label{proof:thm1eq5}
	\lambda_{t+1} &\le \Big(1+ \frac{\lambda_t C_M}{2}\Big) \frac{\sigma_t + \frac{\lambda_t C_M}{2} }{1 + \sigma_t} \lambda_t\le \Big(1+ \frac{\lambda_t C_M}{2}\Big) \Big(\sigma_t + 2 d C_M \lambda_t\Big) \lambda_t.
\end{align}

Consider the term $\sigma_t + 2 d C_M \lambda_t$ in the bound of \eqref{proof:thm1eq5}. We use \textbf{induction} to prove the following statement 
\begin{align}\label{proof:thm1eq6}
	\sigma_t + 2 d C_M \lambda_t \le \Phi_t
\end{align}
for any iteration $t \ge 0$, where
\begin{align}\label{proof:thm1eq7}
	\Phi_t := \prod_{k=1}^{t}\big(1\!-\!\frac{\mu}{\beta_{k, \tau} d L}\big) e^{2(2d + 1) C_M \sum_{k=0}^{t-1}\lambda_k}\frac{dL}{\mu}
\end{align}
For the initial iteration $t=0$, it holds that
\begin{equation}\label{proof:thm1eq8}
	\begin{split}
		\sigma_0 + 2 d C_M \lambda_0 &= {\rm Tr}(\nabla^2 f(\bbx_0)^{-1}\bbB_0) - d + 2 d C_M \lambda_0 \\
		& \le \frac{L}{\mu}{\rm Tr}\big(\nabla^2 f(\bbx_0)^{-1}\nabla^2 f(\bbx_0)\big) - d + 2 d C_M \lambda_0 = \Big(\frac{L}{\mu} - 1\Big) d + 2 d C_M \lambda_0 
	\end{split}
\end{equation}
where the initial condition of $\bbB_0$ is used in the second inequality. By substituting the initial condition of $\lambda_0$ into \eqref{proof:thm1eq8}, we get
\begin{equation}\label{proof:thm1eq85}
	\begin{split}
		\sigma_0 + 2 d C_M \lambda_0 \le \Big(\frac{L}{\mu} - 1\Big) d + \frac{d \ln 2}{2(2d+1)} \le \frac{d L}{\mu}.
	\end{split}
\end{equation}
Thus, \eqref{proof:thm1eq6} holds for $t=0$. Assume that \eqref{proof:thm1eq6} holds for iteration $t \ge 0$, i.e.,
\begin{align}\label{proof:thm1eq9}
	\sigma_{t} + 2 d C_M \lambda_t \le \Phi_t
\end{align}
and consider iteration $t+1$. By substituting \eqref{proof:thm1eq9} into \eqref{proof:thm1eq5} and using the inequality $1+x \le e^x$, we have
\begin{align}\label{proof:thm1eq10}
	\lambda_{t+1} &\le \Big(1+ \frac{\lambda_t C_M}{2}\Big) \Phi_t \lambda_t \le e^{\frac{\lambda_t C_M}{2}} \Phi_t \lambda_t \le e^{2 \lambda_t C_M} \Phi_t \lambda_t.
\end{align}
By using Lemma \ref{lemma:2} that $\phi_t \le \lambda_t$ and Proposition 3, we have 
\begin{align}\label{proof:thm1eq11}
	\sigma_{t+1} &\le \Big(1 - \frac{\mu}{\beta_{t+1,\tau} d L}\Big)(1 + \phi_t C_M)^2\Big(\sigma_t + \frac{2d\phi_t C_M}{1+\phi_t C_M} \Big) \\
	& \le \Big(1 - \frac{\mu}{\beta_{t+1,\tau} d L}\Big)(1 + \lambda_t C_M)^2\Big(\sigma_t + \frac{2d\phi_t C_M}{1+\phi_t C_M} \Big) \nonumber \\
	& \le \Big(1 - \frac{\mu}{\beta_{t+1,\tau} d L}\Big)(1 + \lambda_t C_M)^2\Big(\sigma_t + 2d\phi_t C_M \Big) \le \Big(1 - \frac{\mu}{\beta_{t+1,\tau} d L}\Big)e^{2 C_M \lambda_t} \Phi_t \nonumber
\end{align}
where $\beta_{t+1,\tau}$ is the minimal condition number of the subset $\{\bbe_i\}_{i=1}^\tau$ at iteration $t+1$ and the inequality $(1+x)^2 \le e^{2x}$ is used in the last inequality. Since $1 - \mu/(\beta_{t+1,\tau} d L) \ge 1/2$ with $d \ge 2$, combining \eqref{proof:thm1eq10} and \eqref{proof:thm1eq11} yields
\begin{align}\label{proof:thm1eq12}
	\sigma_{t+1} + 2 d C_M \lambda_{t+1}&\le (1 - \frac{\mu}{\beta_{t+1, \tau} d L})e^{2 C_M \lambda_t} \Phi_t + 2 d C_M e^{2 C_M \lambda_t} \Phi_t \lambda_t \\
	& \le (1 - \frac{\mu}{\beta_{t+1, \tau} d L}) e^{2 C_M \lambda_t} \Phi_t +  (1 - \frac{\mu}{\beta_{t+1, \tau} d L}) 4 d C_M e^{2 C_M \lambda_t} \Phi_t \lambda_t \nonumber\\
	& = (1 - \frac{\mu}{\beta_{t+1, \tau} d L}) e^{2 C_M \lambda_t} (1 + 4 d C_M \lambda_t) \Phi_t. \nonumber
\end{align}
By using the inequality $1+x \le e^x$ and substituting the representation of $\Phi_t$ into \eqref{proof:thm1eq12}, we get
\begin{align}\label{proof:thm1eq13}
	\sigma_{t+1} + 2 d C_M \lambda_{t+1} \le (1 - \frac{\mu}{\beta_{t+1, \tau} d L}) e^{2 C_M (2d + 1) \lambda_t}\Phi_t = \Phi_{t+1}. 
\end{align}
Thus, \eqref{proof:thm1eq6} holds for $t+1$. By combining \eqref{proof:thm1eq85}, \eqref{proof:thm1eq9} and \eqref{proof:thm1eq13}, we prove \eqref{proof:thm1eq6} by induction.

By substituting \eqref{proof:thm1eq6} and the representation of $\Phi_t$ into \eqref{proof:thm1eq10}, we have
\begin{align}\label{proof:thm1eq14}
	\lambda_{t+1} &\le e^{2 \lambda_t C_M} \Phi_t \lambda_t \le e^{2 (2d + 1) \lambda_t C_M} \Phi_t \lambda_t = \frac{1}{\big(1\!-\!\frac{\mu}{\beta_{t+1, \tau} d L}\big)}\Phi_{t+1} \lambda_t.
\end{align}
By using Proposition 2, we get
\begin{align}\label{proof:thm1eq15}
	\Phi_{t+1} &= \prod_{k=1}^{t+1}\big(1\!-\!\frac{\mu}{\beta_{k, \tau} d L}\big) e^{2(2d + 1) C_M \sum_{k=0}^{t}\lambda_k}\frac{dL}{\mu} \\
	&\le \prod_{k=1}^{t+1}\big(1\!-\!\frac{\mu}{\beta_{k, \tau} d L}\big) e^{2(2d + 1) C_M \sum_{k=0}^{t}\big(1-\frac{\mu}{2L}\big)^k\lambda_0}\frac{dL}{\mu} \nonumber 
\end{align}
From the fact that $\sum_{k=0}^{t}\big(1-\frac{\mu}{2L}\big)^k \le 2L/\mu$ and the initial condition, we have
\begin{align}\label{proof:thm1eq16}
	\Phi_{t+1}  \le \prod_{k=1}^{t+1}\big(1\!-\!\frac{\mu}{\beta_{k, \tau} d L}\big) e^{\ln 2}\frac{dL}{\mu} = \prod_{k=1}^{t+1}\big(1\!-\!\frac{\mu}{\beta_{k, \tau} d L}\big) \frac{2 dL}{\mu}. 
\end{align}
Substituting \eqref{proof:thm1eq16} into \eqref{proof:thm1eq14} yields
\begin{align}\label{proof:thm1eq17}
	\lambda_{t+1} &\le \prod_{k=1}^{t}\big(1\!-\!\frac{\mu}{\beta_{k, \tau} d L}\big) \frac{2 dL}{\mu} \lambda_t.
\end{align}
Let $t_0$ be such that
\begin{align}
	\prod_{k=1}^{t_0}\big(1\!-\!\frac{\mu}{\beta_{k, \tau} d L}\big) \frac{2 dL}{\mu} \le 1
\end{align}
and we have
\begin{align}\label{proof:thm1eq18}
	\lambda_{t+t_0+1} &\le \prod_{k=1}^{t+t_0}\big(1\!-\!\frac{\mu}{\beta_{k, \tau} d L}\big) \frac{2 dL}{\mu} \lambda_{t+t_0} \le \prod_{k=t_0+1}^{t+t_0}\big(1\!-\!\frac{\mu}{\beta_{k, \tau} d L}\big) \lambda_{t+t_0}
\end{align}
for any $t \ge 0$. By using \eqref{proof:thm1eq18} recursively, we get
\begin{align}\label{proof:thm1eq19}
	\lambda_{t+t_0+1} &\le \prod_{k=t_0+1}^{t+t_0}\big(1\!-\!\frac{\mu}{\beta_{k, \tau} d L}\big) \lambda_{t+t_0} \le \cdots \le \prod_{k=t_0+1}^{t+t_0}\big(1\!-\!\frac{\mu}{\beta_{k, \tau} d L}\big)^{t+t_0+1 - k} \lambda_{t_0}. 
\end{align}
By further using Proposition 2, we complete the proof
\begin{align}\label{proof:superlineareq20}
	\lambda_{t+t_0+1} \le \prod_{k=t_0+1}^{t+t_0}\big(1\!-\!\frac{\mu}{\beta_{k, \tau} d L}\big)^{t+t_0+1 - k} \Big(1\!-\!\frac{\mu}{2L}\Big)^{t_0}  \lambda_{0}.
\end{align}

\section{Proof of Theorem 2}

We start by noting that for any vector $\bbe_i$ and matrix $\bbE$, we have
\begin{align}\label{proof:coro1eq1}
	\lambda_1 \le \bbe_i^\top \bbE \bbe_i \le \lambda_d
\end{align}
where $\lambda_1$ and $\lambda_d$ are the minimal and maximal eigenvalues of $\bbE$. From Definition 1, the relative condition number of $\bbe_i$ w.r.t. $\bbE$ is bounded as
\begin{align}\label{proof:coro1eq2}
	\beta(\bbe_i) = \frac{\max_{1 \le k \le d} \bbe_k^\top \bbE \bbe_k}{\bbe_i^\top \bbE \bbe_i} \le \frac{\lambda_d}{\lambda_1} = \beta
\end{align}
where $\beta$ is the condition number of $\bbE$. By using \eqref{proof:coro1eq2} together with the corollary condition, we have 
\begin{align}\label{proof:coro1eq3}
	\beta_{t, \tau} \le C_\beta
\end{align}
or
\begin{align}\label{proof:coro1eq4}
	\beta_{t, \tau} = \min_{1\le i \le \tau} \beta(\bbe_i) \le \beta_t \le C_\beta
\end{align}
where $\beta_t$ is the condition number of the approximation error matrix at iteration $t$. By further using \eqref{proof:coro1eq3} or \eqref{proof:coro1eq4} in the result of Theorem 1, we have
\begin{align}\label{proof:coro1eq5}
	\lambda_{t+t_0+1} \le \big(1\!-\!\frac{\mu}{C_\beta d L}\big)^{\frac{t(t+1)}{2}} \Big(1\!-\!\frac{\mu}{2L}\Big)^{t_0}  \lambda_{0}
\end{align}
which completes the proof.

\section{Bound on Condition Number $\beta_t$}

We establish an upper bound on the condition number $\beta_{t}$ of the error matrix $\hat{\bbB}_t - \nabla^2 f(\bbx_{t+1})$ with a minor modification in the correction strategy in Section 3.1. Specifically, we consider the correction strategy on the Hessian approximation $\bbB_t$ as
\begin{align}\label{eq:modifiedCorrection}
    \hat{\bbB}_t \!=\! \big(1\!+\! (\phi_t C_M + \delta_t)\big)\bbB_t
\end{align}
where $\delta_t = q^t 
\delta_0 > 0$ with $\delta_0$ a positive constant and $0<q<1$ a contraction factor. Since $(1\!+\! \phi_t C_M)\bbB_t \succeq \nabla^2 f(\bbx_{t+1})$ from \eqref{eq:lemma31} and $\bbB_t \succeq \nabla^2 f(\bbx_{t})$ from \eqref{eq:lemma32} in Appendix \ref{appendixB}, we obtain
\begin{align}\label{eq:modifiedCorrection1}
    \hat{\bbB}_t - \nabla^2 f(\bbx_{t+1}) \succeq \delta_t \bbB_t \succeq \delta_t \nabla^2 f(\bbx_{t}).
\end{align}
By using Assumption 1 with $L \bbI \succeq \nabla^2 f(\bbx_{t}) \succeq \mu \bbI$ in \eqref{eq:modifiedCorrection1}, we get
\begin{align}\label{eq:lower}
    \hat{\bbB}_t - \nabla^2 f(\bbx_{t+1}) \succeq \delta_t \mu \bbI.
\end{align}
From \eqref{proof:prop2eq0} in Appendix \ref{appendixB}, we have 
\begin{align}\label{eq:modifiedCorrection2}
    \bbB_t \preceq \eta_t \nabla^2 f(\bbx_{t})
\end{align}
where $\eta_t$ is now defined as
\begin{align}\label{eq:modifiedeta}
    \eta_t = e^{2C_M \sum_{k=0}^{t-1}\lambda_k + 2 \sum_{k=0}^{t-1} \delta_t}\frac{L}{\mu}
\end{align}
with the modified correction strategy [cf. \eqref{eq:modifiedCorrection}]. Since both $\{\lambda_t\}_t$ and $\{\delta_t\}_t$ are decreasing geometric sequences, $\eta_t$ is upper bounded by a constant $C_\eta$. By using this fact in \eqref{eq:modifiedCorrection2} and the latter in \eqref{eq:modifiedCorrection}, we have
\begin{align}\label{eq:upper}
    \hat{\bbB}_t \preceq \big(1\!+\! (\phi_t C_M + \delta_t)\big) C_\eta \nabla^2 f(\bbx_{t}) \preceq \big(1\!+\! (\lambda_0 C_M + \delta_0)\big) C_\eta L \bbI
\end{align}
where $\lambda_t \le \lambda_0$, $\delta_t \le \delta_0$ and $\mu \bbI \preceq \nabla^2 f(\bbx_{t}) \preceq L \bbI$ are used in the second inequality. By using \eqref{eq:lower} and \eqref{eq:upper}, we can bound the condition number $\beta_t$ of $\hat{\bbB}_t - \nabla^2 f(\bbx_{t+1})$ as
\begin{align}\label{eq:upperBound}
     \beta_t \le \frac{\big(2\!+\! (\lambda_0 C_M + \delta_0)\big) C_\eta L}{q^t 
     \delta_0 \mu} = C_{t,\beta}
\end{align}
at iteration $t$, where $\lambda_0$ and $\delta_0$ are initial constants. We remark that the upper bound in \eqref{eq:upperBound} is the worst-case analysis because it holds for the condition number $\beta_t$, i.e., the minimal relative condition number $\beta_{t,1}$ with memory size $\tau = 1$ [Def. 1]. Therefore, it is important to note that this bound may not be tight and the actual value of $\beta_{t,\tau}$ could be smaller.

\begin{table}
\caption{Details of datasets: svmguide3, connect-4, protein and mnist.}
\label{tab:dataset}
\centering
\begin{tabular}{llllll}
	\toprule
	Dataset & Number of samples $N$ & Feature dimension $d$ & Regularization parameters $\mu$ \\
	\midrule
	Svmguide3 & ~~~~~~~~~~~~1243 & ~~~~~~~~~~~~21 & ~~~~~~~~~~~~~~~~$10^{-4}$ \\
	Connect-4 & ~~~~~~~~~~~~67557 & ~~~~~~~~~~~~126 & ~~~~~~~~~~~~~~~~$10^{-4}$ \\
	Protein & ~~~~~~~~~~~~17766 & ~~~~~~~~~~~~357 & ~~~~~~~~~~~~~~~~$10^{-4}$ \\
	Mnist & ~~~~~~~~~~~~60000 & ~~~~~~~~~~~~780 & ~~~~~~~~~~~~~~~~$10^{-6}$ \\
	\bottomrule
\end{tabular}
\end{table}

\begin{figure}[t]
\centering
\begin{subfigure}{0.42\linewidth}
	\includegraphics[width=\linewidth]{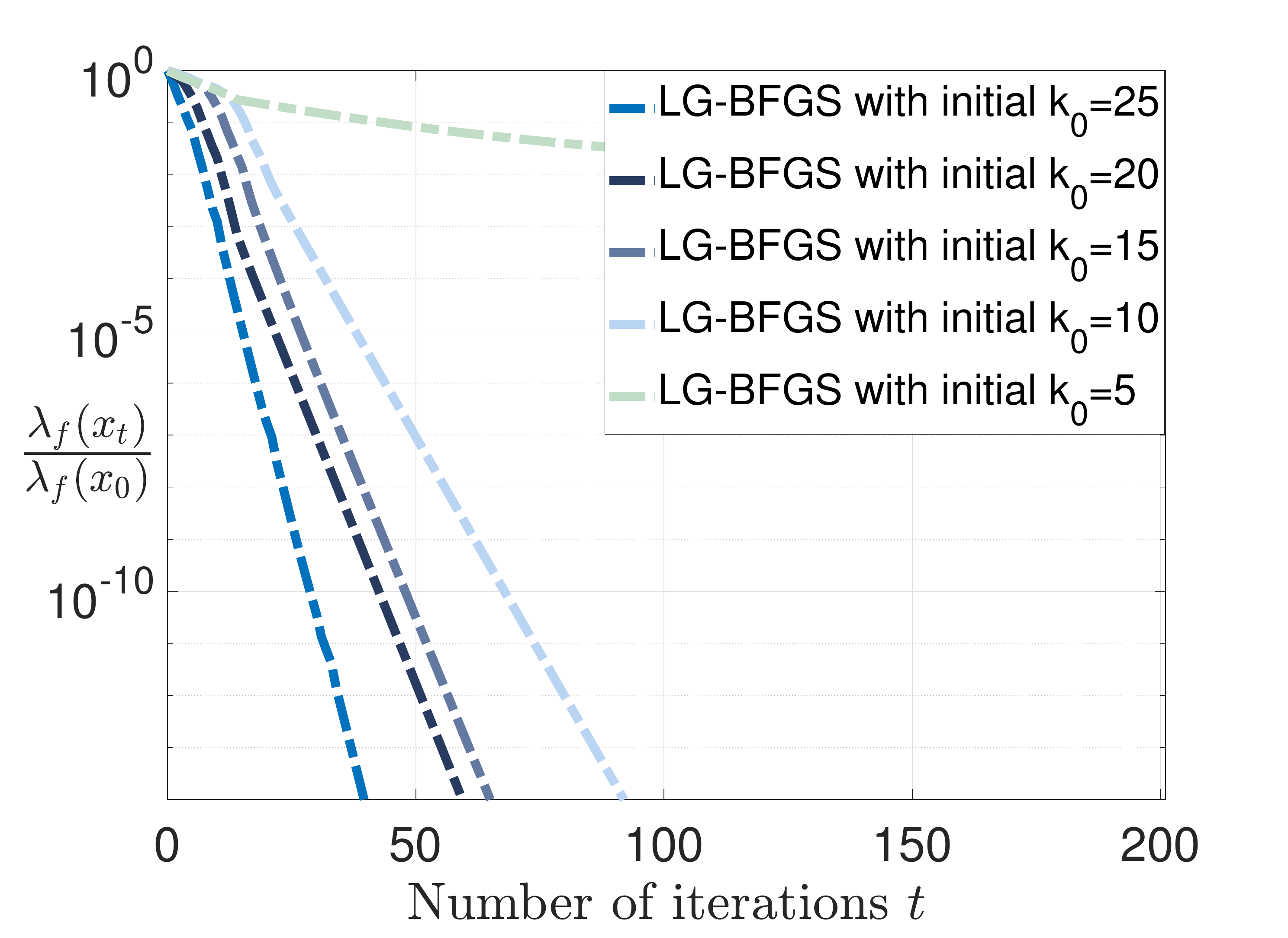}
	\caption{Svmguide3 dataset}
\end{subfigure}
\begin{subfigure}{0.42\linewidth}
	\includegraphics[width=\linewidth]{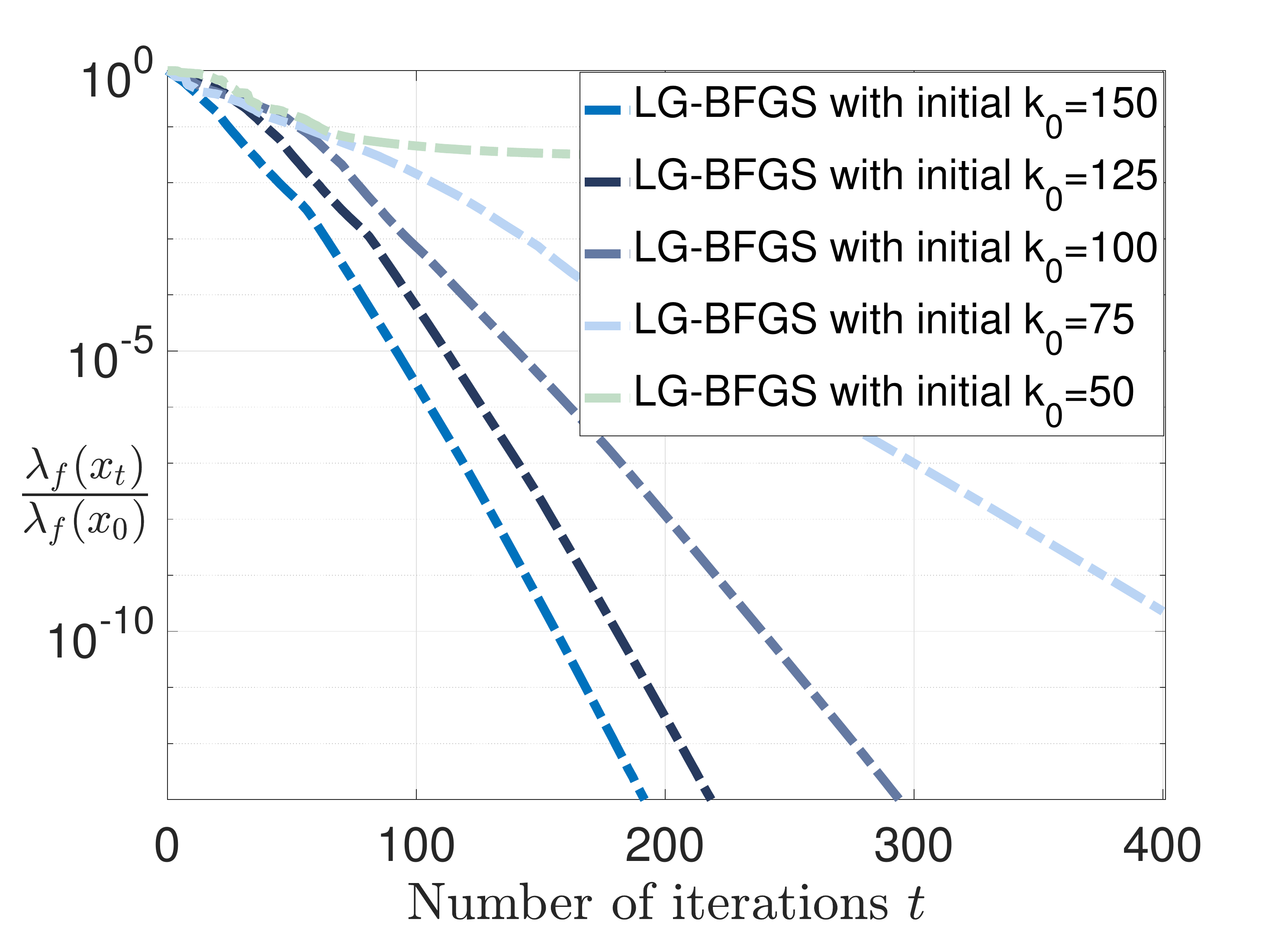}
	\caption{Connect-4 dataset}
\end{subfigure}
\begin{subfigure}{0.42\linewidth}
	\includegraphics[width=\linewidth]{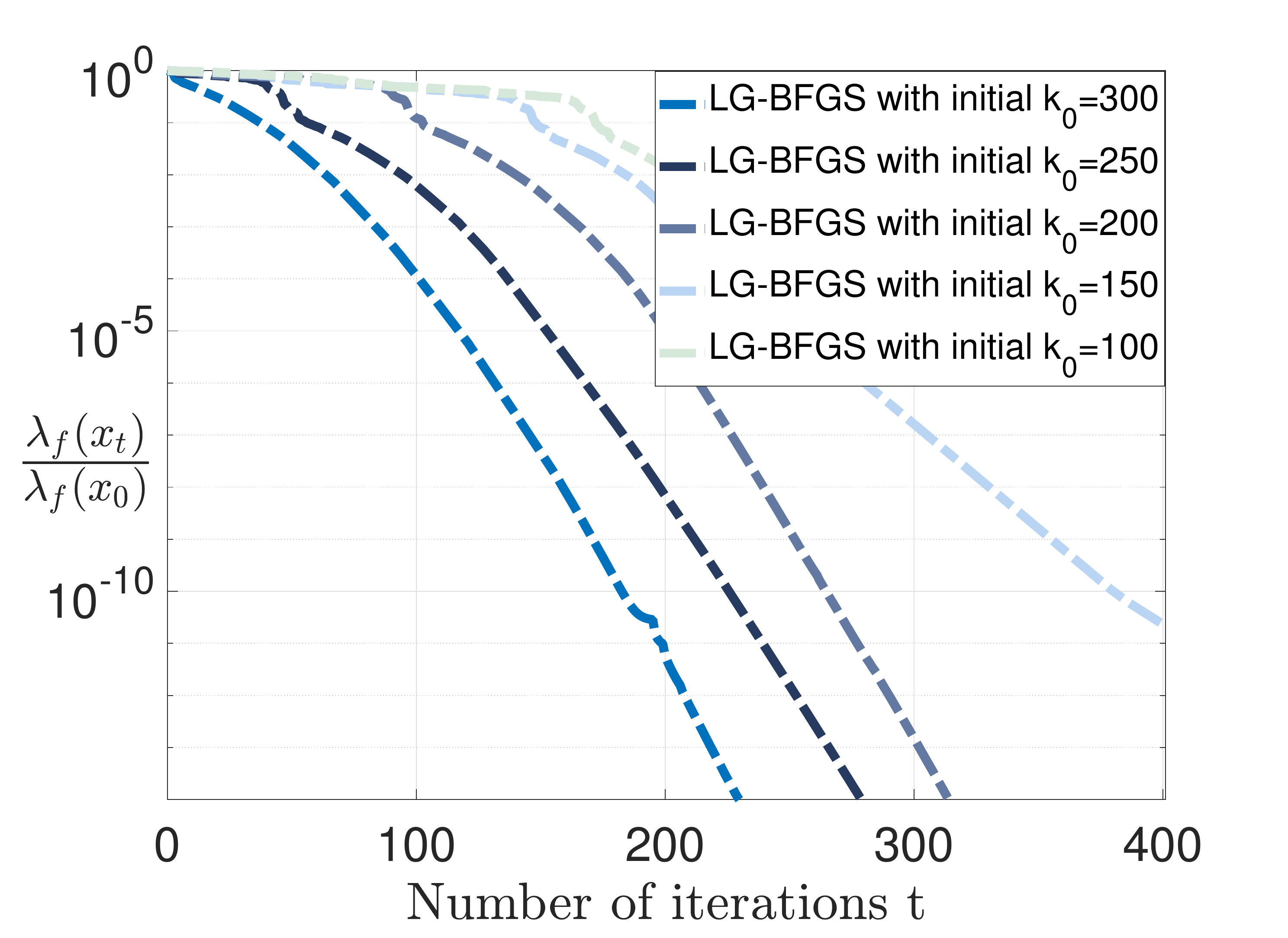}
	\caption{Protein dataset}
\end{subfigure}
\begin{subfigure}{0.42\linewidth}
	\includegraphics[width=\linewidth]{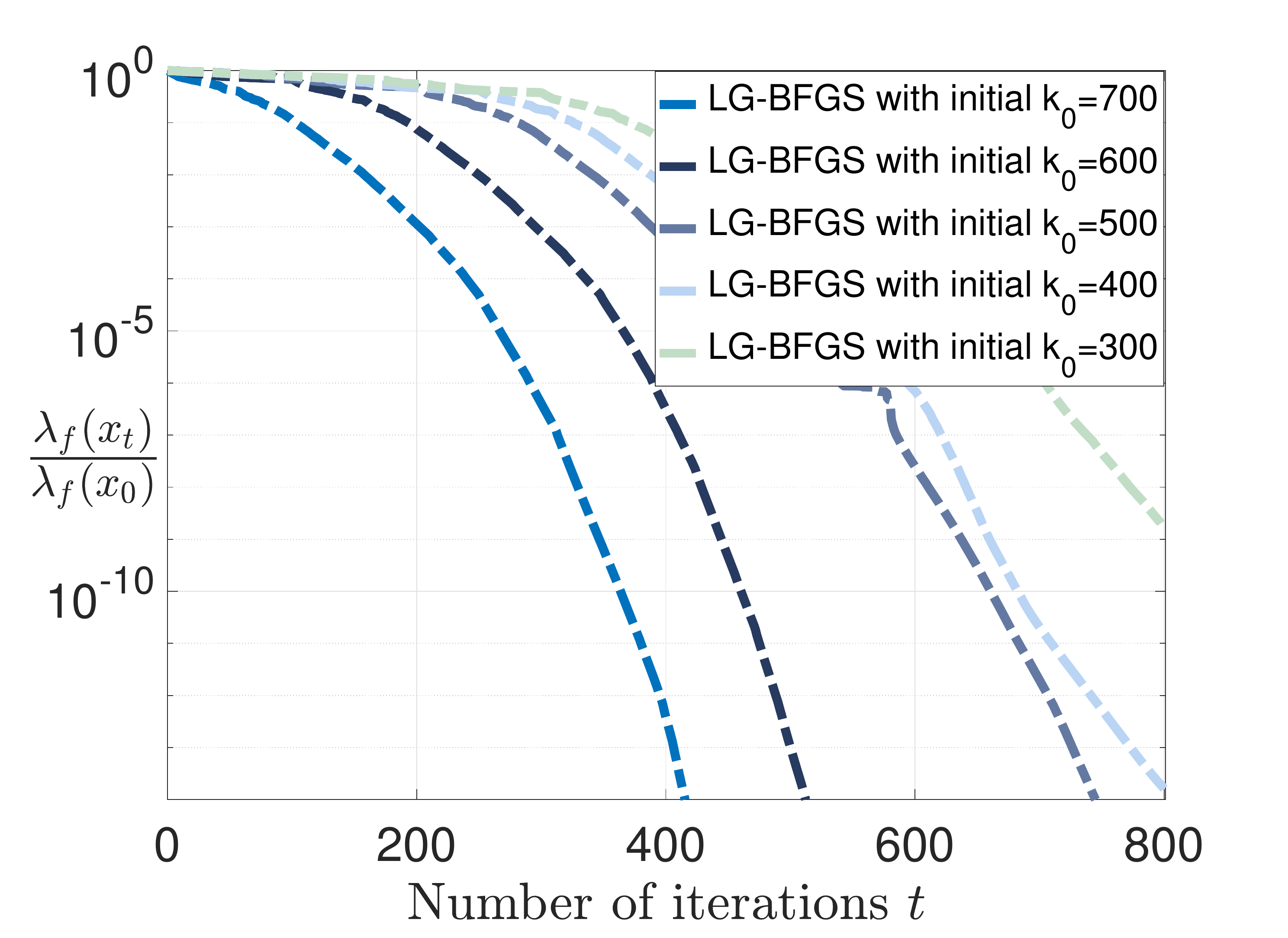}
	\caption{Mnist dataset}
\end{subfigure}
\caption{Performance of LG-BFGS with different initialization on four datasets.}
\label{fig:2}
\end{figure}

By following similar steps as in the proofs of Proposition 2 - Theorem 2, we can establish an explicit convergence rate of LG-BFGS as
\begin{equation}\label{eq:modifiedCorrection3}
		\lambda_f(\bbx_{t+t_0+1}) \le \prod_{u=t_0+1}^{t+t_0}\Big(1\!-\!\frac{\mu}{C_{u,\beta}dL}\Big)^{t+t_0+1-u} 
  \Big(1\!-\!\frac{\mu}{2L}\Big)^{t_0} \lambda_f(\bbx_{0})
\end{equation}
where the modified correction strategy may require a slightly more accurate initialization to derive this rate. Since this upper bound $C_{t,\beta}$ is not constant but increases with iteration $t$ [cf. \eqref{eq:upperBound}], the convergence rate in \eqref{eq:modifiedCorrection3} is slower than that in Theorem 2. Specifically, we can represent $C_{t, \beta}$ in \eqref{eq:upperBound} as the form of
\begin{align}\label{eq:modifiedCorrection4}
    C_{t,\beta} = C_\beta q^{-t}
\end{align}
where 
$C_\beta$ is a constant. By substituting \eqref{eq:modifiedCorrection4} into \eqref{eq:modifiedCorrection3}, we get
\begin{equation}\label{eq:modifiedCorrection5}
		\lambda_f(\bbx_{t+t_0+1}) \le \prod_{u=t_0+1}^{t+t_0}\Big(1\!-\!\frac{\mu}{C_{\beta}dL}q^u\Big)^{t+t_0+1-u} 
  \Big(1\!-\!\frac{\mu}{2L}\Big)^{t_0} \lambda_f(\bbx_{0}).
\end{equation}
We can approximate \eqref{eq:modifiedCorrection5} as
\begin{align}\label{eq:modifiedCorrection6}
    &\prod_{u=t_0+1}^{t+t_0}\Big(1\!-\!\frac{\mu}{C_{\beta}dL}q^u\Big)^{t+t_0+1-u} 
  \Big(1\!-\!\frac{\mu}{2L}\Big)^{t_0} \lambda_f(\bbx_{0}) \\
  &\approx e^{-\sum_{u=t_0+1}^{t+t_0} \frac{\mu}{C_\beta d L} (t+t_0+1-u) q^u} \Big(1\!-\!\frac{\mu}{2L}\Big)^{t_0} \lambda_f(\bbx_{0}) \nonumber \\
    & = e^{-\frac{q^{t_0+1}\mu}{C_\beta d L} \sum_{u=0}^{t-1} (t-u) q^u} \Big(1\!-\!\frac{\mu}{2L}\Big)^{t_0} \lambda_f(\bbx_{0}) \le e^{- C t} \Big(1\!-\!\frac{\mu}{2L}\Big)^{t_0} \lambda_f(\bbx_{0}) \nonumber
\end{align}
where $C=q^{t_0+1}\mu/(C_\beta d L)$ is a constant. By combining \eqref{eq:modifiedCorrection6} and the result in Proposition 2, we have
\begin{equation}\label{eq:modifiedCorrection7}
		\lambda_f(\bbx_{t+t_0+1}) \le \min \Big\{ e^{- C t} \Big(1-\frac{\mu}{2L}\Big)^{t_0} \lambda_f(\bbx_{0}), \Big(1-\frac{\mu}{2L}\Big)^{t+t_0+1} \lambda_f(\bbx_{0}) \Big\}.
\end{equation}
This can be considered as an improved linear rate 
depending on specific problem settings.

\section{Additional Experiments} 

\begin{figure}[t]
\centering
\begin{subfigure}{0.42\linewidth}
	\includegraphics[width=\linewidth]{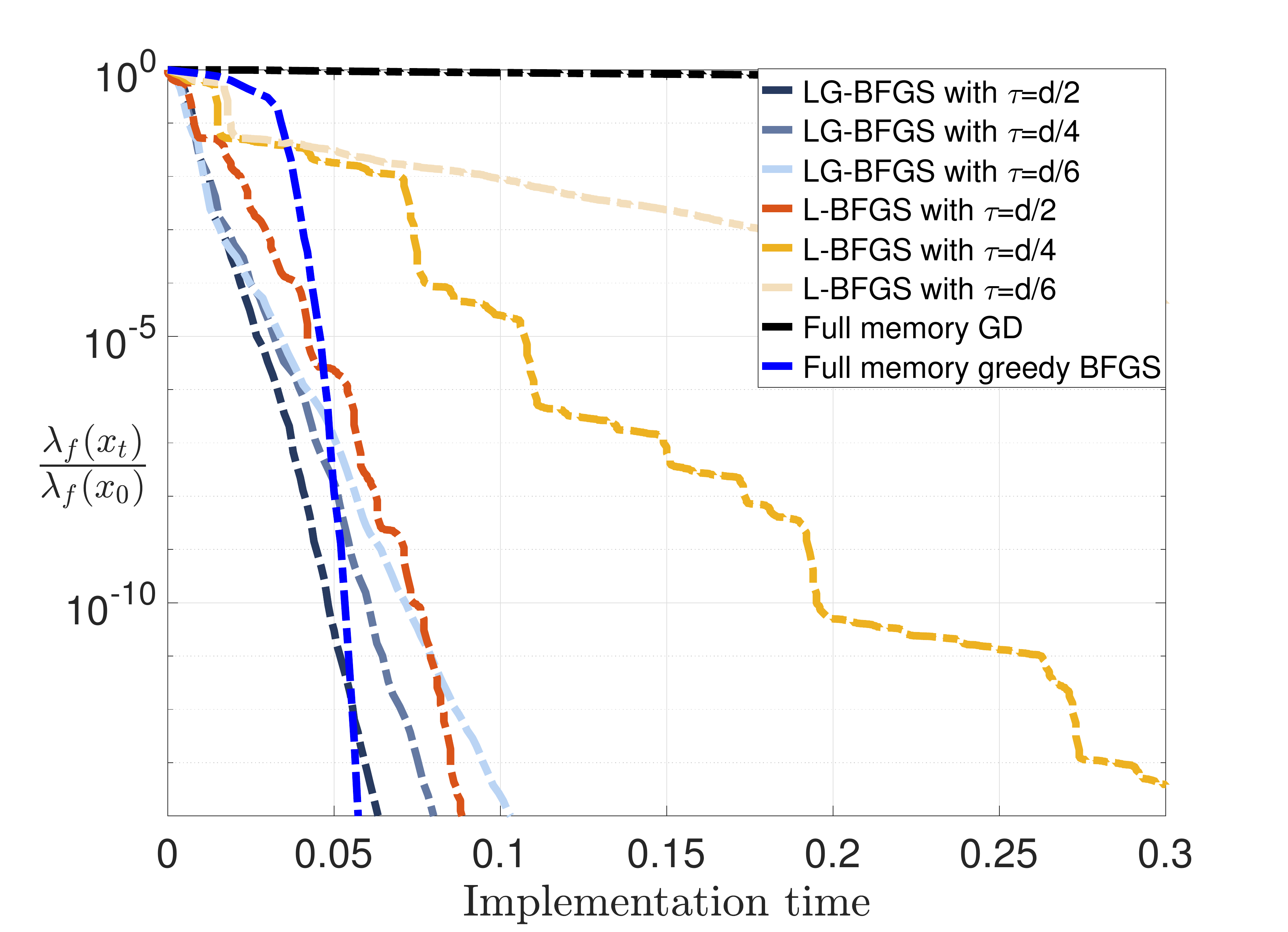}
	\caption{Svmguide3 dataset}
\end{subfigure}
\begin{subfigure}{0.42\linewidth}
	\includegraphics[width=\linewidth]{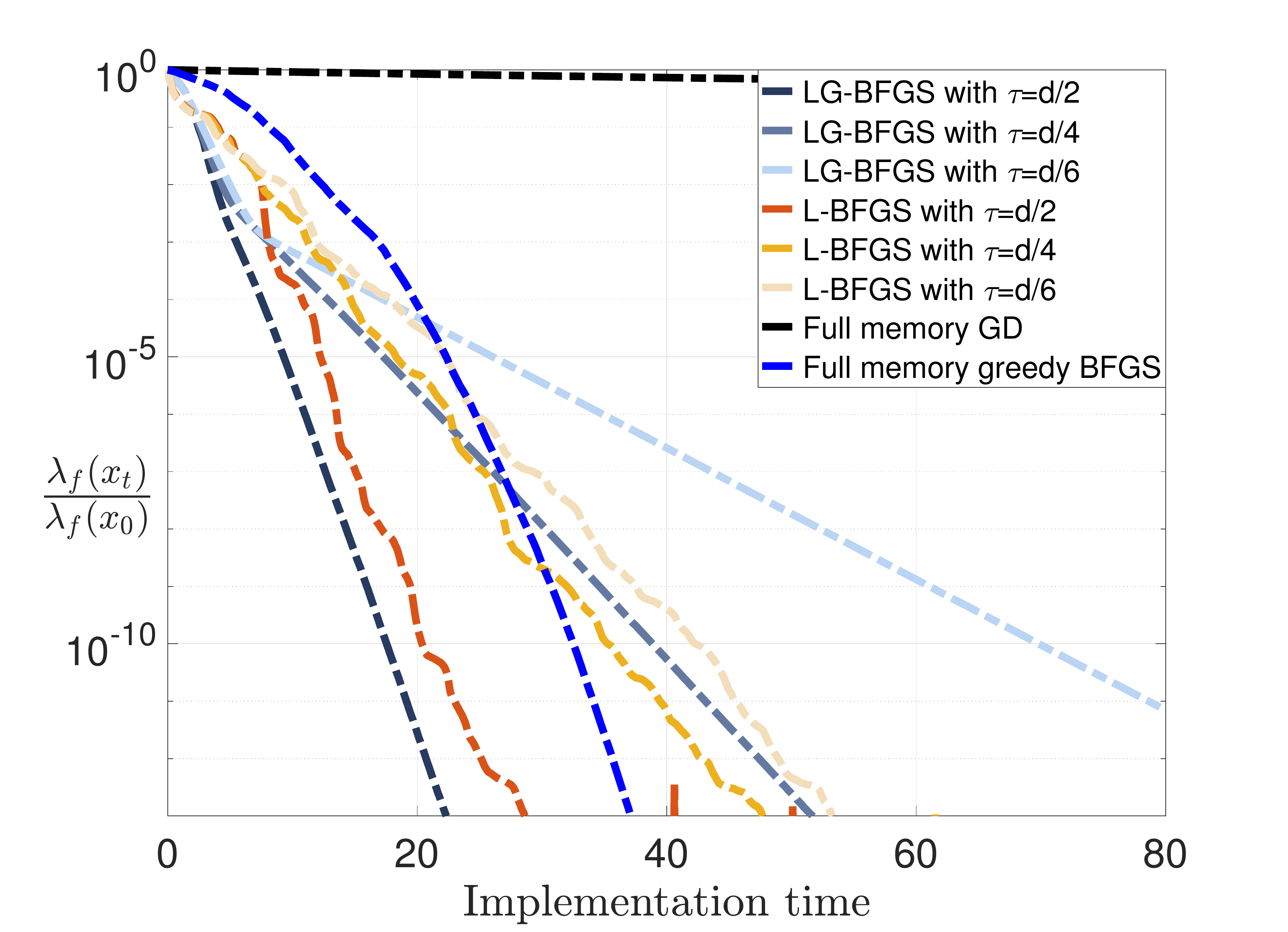}
	\caption{Connect-4 dataset}
\end{subfigure}
\begin{subfigure}{0.42\linewidth}
	\includegraphics[width=\linewidth]{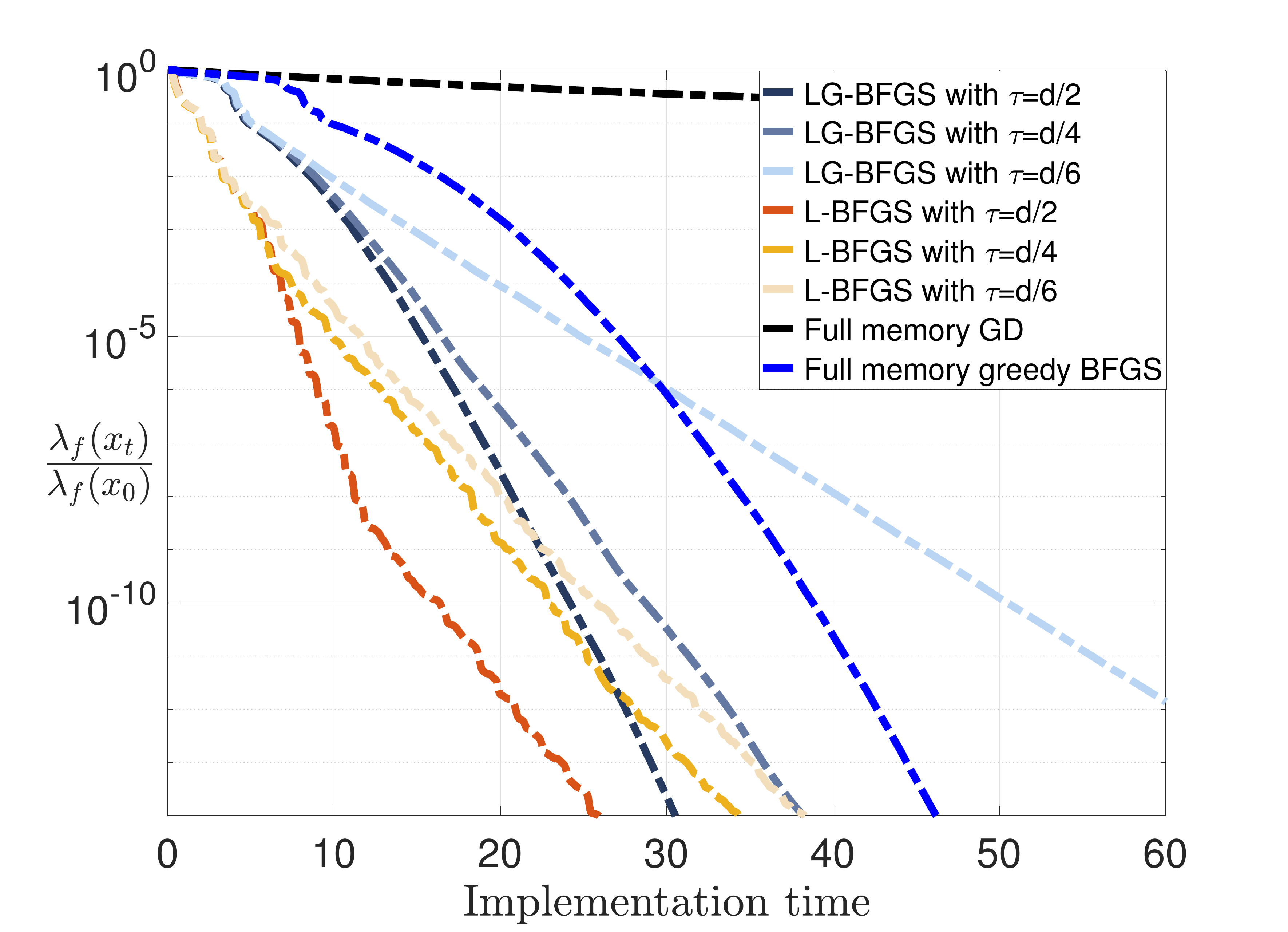}
	\caption{Protein dataset}
\end{subfigure}
\begin{subfigure}{0.42\linewidth}
	\includegraphics[width=\linewidth]{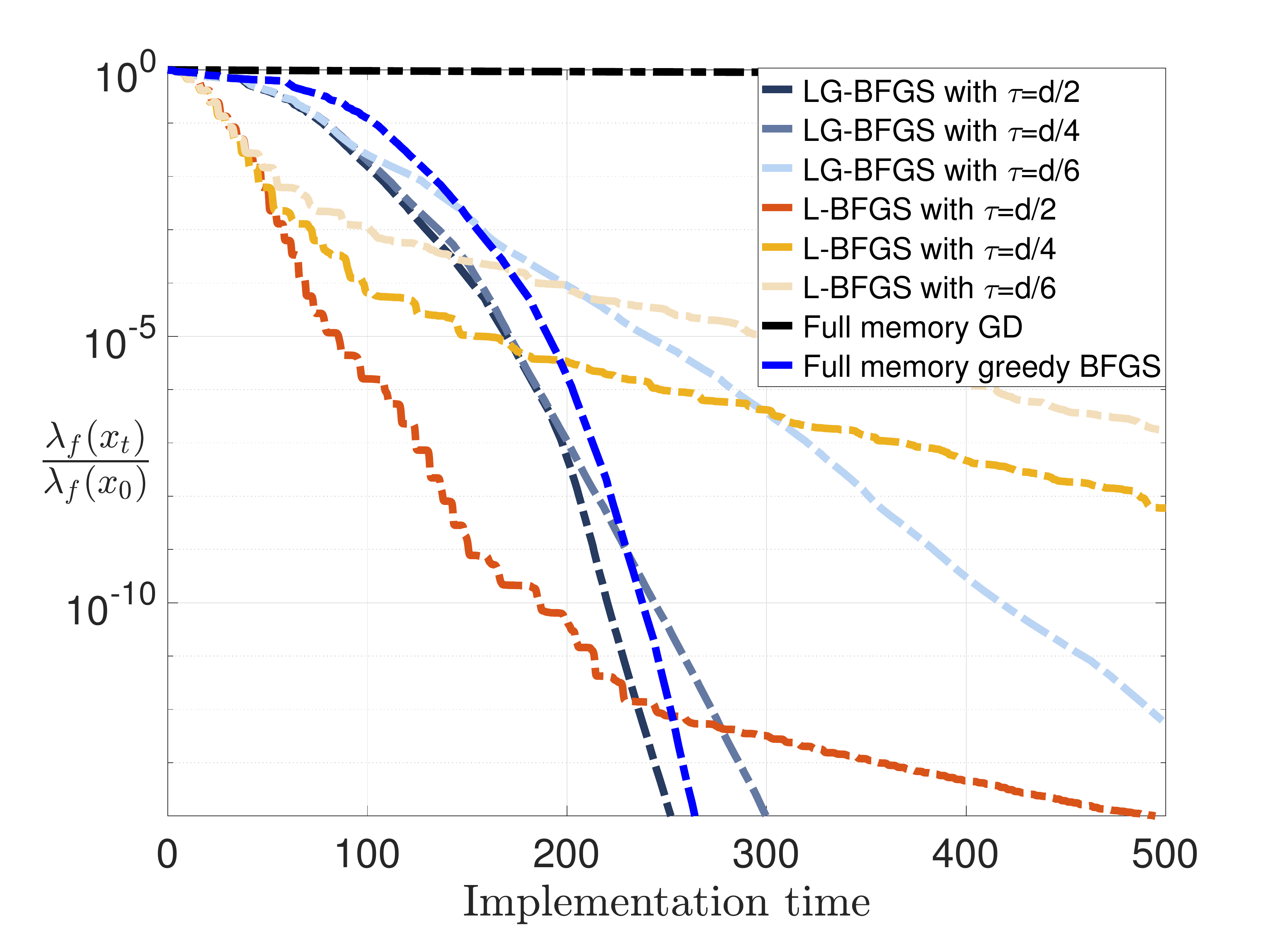}
	\caption{Mnist dataset}
\end{subfigure}
\caption{Performance of LG-BFGS, L-BFGS, and greedy BFGS over implementation time on four datasets. We consider different memory sizes for LG-BFGS and L-BFGS.}
\label{fig:3}
\end{figure} 

We consider four datasets: svmguide3, connect-4, protein and mnist for classification problem, details of which are summarized in Table \ref{tab:dataset}. With the local nature of superlinear convergence results for quasi-Newton methods, we construct a setup with a warm start for all methods, i.e., the initialization is close to the solution by performing greedy BFGS for $k_0$ iterations. This has the practical effect of reducing the superlinear phase triggering time -- see \citep{jin2022sharpened} for further details. Also worth mentioning is that we found it is better not to apply the correction strategy in LG-BFGS and greedy BFGS methods in practice following \citep{rodomanov2021greedy,lin2021greedy,jin2022sharpened}, i.e., simply set $\tilde{\bbr}_u = \bbr_u$ in step 2 of Algorithm 1 for the displacement step of LG-BFGS and $\hat{\bbB}_t = \bbB_t$ in the Hessian approximation update of greedy BFGS.

Fig. \ref{fig:2} evaluates LG-BFGS with different initialization. We see that the performance of LG-BFGS increases with the improvement of initialization in all experiments. This relationship is expected because (i) the superlinear convergence of LG-BFGS is a local result; and (ii) the subset $\{\bbe_i\}_{i=1}^\tau$ being selected from good initialization roughly ensures the update progress of the Hessian approximation associated with the sparse subspace, i.e., the minimal relative condition number $\beta_\tau$ w.r.t. the approximation error matrix in (17) is small. These aspects manifest in the improved 
convergence 
of LG-BFGS, which corroborate our theoretical findings in Section 4.

Fig. \ref{fig:3} shows the convergence of LG-BFGS, L-BFGS, greedy BFGS and GD as a function of implementation time. For greedy BFGS, it has the fastest convergence rate (per iteration) but requires the most computational cost, 
which slows its convergence in datasets of connect-4, protein and mnist. For L-BFGS, it requires the lowest computational cost but has the slowest convergence rate, which exhibits bad performance with small memory sizes in datasets of svmguide3 and mnist. LG-BFGS strikes a balance between convergence rate of greedy BFGS and computational cost of L-BFGS, i.e., it requires less computational cost than the former and obtains a faster convergence rate than the latter, corresponding to our discussions in Section 5. A final comment is that LG-BFGS and L-BFGS require less storage memory $\ccalO(\tau d)$ than that required by greedy BFGS $\ccalO(d^2)$. 

\end{document}